\documentclass[12pt]{article}

\usepackage{enumerate}
\usepackage{amsthm}
\usepackage{amssymb}
\usepackage{amsmath}
\usepackage{amsfonts}
\usepackage{cancel}
\usepackage{times}
\usepackage{bbm}
\usepackage[noadjust]{cite}
\usepackage{color, soul}
\usepackage{hyperref}
\usepackage{wasysym}
\hypersetup{
     colorlinks   = true,
     citecolor    = black,
     linkcolor    = blue
}

\usepackage[paper=a4paper, top=2.5cm, bottom=2.5cm, left=2.5cm, right=2.5cm]{geometry}

\def\qed{\hfill $\vcenter{\hrule height .3mm
\hbox {\vrule width .3mm height 2.1mm \kern 2mm \vrule width .3mm
height 2.1mm} \hrule height .3mm}$ \bigskip}

\def \RR {\mathbb R}
\def \NN {\mathbb N}
\def \EE {\mathbb E}
\def \ZZ {\mathbb Z}

\def \PP {\mathbb P}
\def \eps {\varepsilon}

\newtheorem{theorem}{Theorem}
\newtheorem{lemma}{Lemma}

\newtheorem{proposition}{Proposition}
\newtheorem{corollary}[theorem]{Corollary}
\theoremstyle{definition}
\newtheorem{definition}[theorem]{Definition}
\theoremstyle{remark}
\newtheorem{remark}[theorem]{Remark}
\newtheorem{example}[theorem]{Example}
\newtheorem*{remark*}{Remark}

\long\def\symbolfootnotetext[#1]#2{\begingroup
\def\thefootnote{\fnsymbol{footnote}}\footnotetext[#1]{#2}\endgroup}

\begin{document}
\title{Anti-concentration of polynomials: dimension-free covariance bounds and decay of
Fourier coefficients}
\author{Itay Glazer\\
	Northwestern University \and Dan Mikulincer \thanks{DM is partially supported by a European Research Council grant no. 803084}\\
Weizmann Institute of Science}
\date{\today}

\maketitle
\global\long\def\NN{\mathbb{N}}%
\global\long\def\RR{\mathbb{\mathbb{R}}}%
\global\long\def\PP{\mathbb{\mathbb{P}}}%
\global\long\def\eps{\mathbb{\varepsilon}}%
\global\long\def\EE{\mathbb{E}}%
\global\long\def\ZZ{\mathbb{Z}}%

\begin{abstract}
We study random variables of the form $f(X)$, when $f$ is a degree
$d$ polynomial, and $X$ is a random vector on $\RR^{n}$, motivated
towards a deeper understanding of the covariance structure of $X^{\otimes d}$.
For applications, the main interest is to bound $\mathrm{Var}(f(X))$
from below, assuming a suitable normalization on the coefficients
of $f$. Our first result applies when $X$ has independent coordinates,
and we establish dimension-free bounds. We also show that the assumption
of independence can be relaxed and that our bounds carry over to uniform
measures on isotropic $L_{p}$ balls. Moreover, in the case of the
Euclidean ball, we provide an orthogonal decomposition of $\mathrm{Cov}(X^{\otimes d})$.
Finally, we utilize the connection between anti-concentration and
decay of Fourier coefficients to prove a high-dimensional analogue
of the van der Corput lemma, thus partially answering a question posed
by Carbery and Wright. 
\end{abstract}

\section{Introduction}

Let $X\sim\mu$ be a random vector in $\RR^{n}$. Fix $d\in\NN$ and
consider the tensor power $X^{\otimes d}$, which is a random vector
in $(\RR^{n})^{\otimes d}$. The main motivation for this present
study came from trying to understand the spectrum of the covariance
matrix, $\mathrm{Cov}(X^{\otimes d})$. This question has lately gained
interest in the study of central limit theorems for tensor powers
(\cite{mikulincer2021clt,nourdin2018asymptotic,nourdin2021gaussian,bubeck2018entropic,fang2020new})
with connections to random geometric graphs (\cite{bubeck2016testing,brennan2020phase,brennan2021finetti,liu2021phase}
and universality of neural networks (\cite{eldan2021non}).

Specifically, we are interested in identifying regimes where the smallest,
non-trivial, eigenvalue of $\mathrm{Cov}(X^{\otimes d})$ can be bounded
from below in a \emph{dimension-free} way. We remark that corresponding
bounds for the largest eigenvalue can be proved in a straightforward
manner using standard concentration techniques, and that, typically,
one cannot expect to obtain dimension-free bounds (see \cite[Lemma 4]{eldan2021non}
and the remark that follows).

Observe that $\mathrm{Cov}(X^{\otimes d})$ is an $n^{d}\times n^{d}$
matrix, which is necessarily singular due to symmetries. Thus, we
slightly abuse notations and consider $X^{\otimes d}$ as a random
element in $\mathrm{Sym}_{d}(\RR^{n})$, the subspace of symmetric
tensors. Note that even if $\mathrm{Cov}(X)$ is simple, say if $X$
is isotropic, $\mathrm{Cov}(X^{\otimes d})$ can be quite complicated
because of the introduced dependencies.\\

To set the stage for our results, we now rephrase the problem as a
question about anti-concentration of polynomials. Introduce the multi-indices
$(I_{1},\dots,I_{n})=I\in\NN^{n}$, for which we use the standard
multi-index notation. For $(x_{1},\dots x_{n})=x\in\RR^{n}$, 
\[
\left|I\right|=\sum\limits _{i=1}^{n}I_{i}\text{ and }x^{I}=\prod_{i=1}^{n}x_{i}^{I_{i}}.
\]

We fix a standard orthonormal basis for $\mathrm{Sym}_{d}(\RR^{n})$,
indexed by the multi-indices, $\{e_{I}\}_{|I|=d}.$ To bound the eigenvalues
of $\mathrm{Cov}(X^{\otimes d})$ from below it will be enough to
show that if $v\in\mathrm{Sym}_{d}(\RR^{n})$ is a unit vector, then
$\mathrm{Var}(\langle v,X^{\otimes d}\rangle)$ is large. Write 
\[
v=\sum\limits _{\left|I\right|=d}v_{I}e_{I}\text{ with }\sum\limits _{\left|I\right|=d}v_{I}^{2}=1.
\]
Let us define the homogeneous degree $d$ polynomial $f:\RR^{n}\to\RR$,
by 
\[
f(x)=\sum\limits _{\left|I\right|=d}v_{I}x^{I}=\langle x^{\otimes d},v\rangle.
\]
Hence, $\mathrm{Var}(\langle v,X^{\otimes d}\rangle)=\mathrm{Var}(f(X))$.
From this perspective, our original question reduces to showing that
if $f$ is a homogeneous polynomial, such that the square of its coefficients
sums to $1$, then $f(X)$ cannot be too concentrated around its expectation.

The phenomenon of anti-concentration is further manifested through
sublevel set and Fourier estimates. A polynomial $f(X)$ which is
not too concentrated around any point is expected to have a low probability
of being contained in a small interval in $\RR$, and to have fast
decay of Fourier coefficients (see the discussion in Section \ref{subsec:Decay-of-Fourier}).
We explore all of the above in this paper.

Our main results are summarized below:
\begin{itemize}
\item We show that if $\mu$ is a product measure, one can bound $\mathrm{Var}(f(X))$
in a way that depends only on the degree $d$ and the marginal of
the measure $\mu$. Moreover, the bound is uniform over isotropic
log-concave measures. The result also applies to non-homogeneous polynomials,
under some appropriate assumption concerning the coefficients of $f$. 
\item To allow some form of dependence, we also consider the case where
$\mu$ is the uniform measure on an isotropic $L_{p}$ ball and obtain
corresponding results. 
\item In case $X$ is uniformly distributed on the isotropic Euclidean ball, or more generally, when $X$ is radially symmetric,
we completely characterize the spectrum of $\mathrm{Cov}(X^{\otimes d})$
and express the eigenvectors in terms of the spherical harmonics. 
\item When specializing to log-concave measures, we also establish sublevel
estimates. Namely, not only is $\mathrm{Var}(f(X))$ large, but for
$\eps>0$, one can control, 
\[
\PP\left(\left|f(X)\right|\leq\eps\right).
\]
\item We apply our results to log-concave product measures and derive a
dimension-free multivariate analogue of the classical van der Corput
lemma for polynomials (cf. \cite[Section 7]{CCW99}). Informally,
let $f$ be a polynomial of degree $d$ that has at least one large
coefficient, which corresponds to a monomial of degree $d$. Then,
if $\mu$ is a log-concave product measure, the Fourier coefficients
of $f_{*}\mu$ decay rapidly. When considering the cube, this gives
a partial answer to a question asked by Carbery and Wright in \cite{carbery2001distributional}. 
\end{itemize}
\paragraph{Acknowledgments: }
We are indebted to an anonymous referee for spotting a mistake in Lemma \ref{lem:logconcavevdc} in an earlier version, as well as for many useful comments.
\section{Main results and related work}

Before stating our main results let us first introduce some notation
and definitions.

\subsection{Definitions, notation and conventions}

If $f:\RR^{n}\to\RR$, $f(x)=\sum\limits _{\left|I\right|\leq d}\alpha_{I}x^{I}$
is a degree $d$ polynomial, we define its \emph{$d$-level content}
as, 
\[
\mathrm{coeff}_{d}(f):=\sqrt{\sum\limits _{\left|I\right|=d}\alpha_{I}^{2}}.
\]
As will become apparent, $\mathrm{coeff}_{d}(f)$ may serve as a scale
parameter to measure the variance of the push-forward measure $f_{*}\mu$.

If $\nu$ is a measure on $\RR$, we will denote by $\nu^{\otimes n}$
its $n$-fold tensor product, which is a product measure on $\RR^{n}$.
We say that a measure on $\RR^{n}$ is \emph{isotropic} if it is centered
and its covariance matrix is the identity. If $\mu$ is of the form
$\mu=e^{-\varphi(x)}dx$ for some convex function $\varphi$, we will
say that $\mu$ is \emph{log-concave}. 

As a convention, an absolute constant will be denoted by $C,C'$,
etc. A constant depending on a given data will be denoted using subscript,
e.g. $C_{d}$ (resp. $C_{d,n}$) is a constant depending only on $d$
(resp. $d$ and $n$). Still, when formulating the main results, to
maximize clarity, we will state the precise dependence of the constants
on the data.

\subsection{Variance bounds for polynomials}

Our first main result deals with product measures. \begin{theorem}
\label{thm:product} Let $\mu$ be a centered measure on $\RR$ whose
support is infinite and let $d\in\NN$. Then:
\begin{enumerate}
\item There exists a constant $C_{\mu,d}$, which depends on $\mu$ and
$d$ only, such that for every $n\in\NN$ and every polynomial $f:\RR^{n}\to\RR$
of degree $d$, 
\[
\mathrm{Var}_{\mu^{\otimes n}}(f)\geq C_{\mu,d}\cdot\mathrm{coeff}_{d}^{2}(f).
\]
\item If $\mu$ is also log-concave and isotropic, one may always take $C_{\mu,d}=\frac{1}{2^{15d}}$. 
\end{enumerate}
\end{theorem} The requirement that $\mu$ has infinite support is
necessary here. Otherwise, one can always choose a polynomial $f$
of degree large enough, so that $f$ vanishes on the support of $\mu^{\otimes n}$.
In which case $\mathrm{Var}_{\mu^{\otimes n}}(f)=0$. Theorem \ref{thm:product} includes in it the standard Gaussian, which
was considered before in \cite[Lemma 5]{eldan2021non}. We recover
this result and actually improve upon the stated constant.

Most of the work on normal approximations for tensor powers revolved
around product measures (see \cite{bubeck2018entropic,fang2020new}).
In this case, Theorem \ref{thm:product} gives a complete dimension-free
picture. Still, the question is also interesting for measures that
do not have a product structure. Let us point out that our proof of
Theorem \ref{thm:product} goes through an orthogonal decomposition
of $L^{2}(\mu^{\otimes n})$, which relies on a particular form taken
by orthogonal polynomials of measures on the real line. Hence, it
is adapted to product measures, and we are not able to apply it in
the general case (however, see \cite{bakry2013orthogonal}, for some
examples of high-dimensional orthonormal polynomials, where our method
could prove useful).

To address the point raised above, we identify one class of non-product
measures where we can derive similar results, the uniform measures
on isotropic $L_{p}$ balls. For $p\geq1$ and $x=(x_{1},\dots,x_{n})\in\RR^{n}$
define its $p$-norm, by $\|x\|_{p}=\left(\sum\limits _{i=1}^{n}\left|x_{i}\right|^{p}\right)^{\frac{1}{p}}$
and define the unit ball of this norm, 
\[
B_{p,n}:=\{x\in\RR^{n}:\|x\|_{p}\leq1\}.
\]
An isotropic $L_{p}$ ball is a re-normalization $\tilde{B}_{p,n}=z_{p,n}B_{p,n}$,
such that the measure $\mathrm{Uniform}(\tilde{B}_{p,n})$ is isotropic.

The uniform measure on $\tilde{B}_{p,n}$ is reminiscent of a product measure.
Specifically, it is a well known fact that
if $Y$ is a random vector in $\RR^{n}$ with a product density, proportional
to $e^{-\|x\|_{p}^{p}}$, then if $U\sim\mathrm{Uniform}([0,1])$
is independent from $Y$, we have that $U^{\frac{1}{n}}\frac{Y}{\|Y\|_{p}}$
is uniformly distributed on $B_{p,n}$ (see \cite{schechtman90volume}).
Coupling this with the previous theorem we then obtain. \begin{theorem}\label{thm:pballs}
Let $p\geq1$, and let $\mu=\mathrm{Uniform}(\tilde{B}_{p,n})$. Fix
$d\in\NN$, then there exists a constant $C_{d}>0$, which depends
only on $d$, such that if $f:\RR^{n}\to\RR$ is a degree $d$ homogeneous
polynomial, 
\[
\EE_{\mu}[f^{2}]\geq C_{d}\mathrm{coeff}_{d}^{2}(f).
\]
\end{theorem} The constant $C_{d}$ in Theorem \ref{thm:pballs}
is explicit. Since it has a somewhat complicated expression,
we chose to present it this way. Whether the same conclusion holds
for general isotropic log-concave measures, maybe with suitable symmetries,
is an interesting question that is left open.

In contrast to Theorem \ref{thm:product}, Theorem \ref{thm:pballs}
is restricted to homogeneous polynomials and only deals with the second
moment, as opposed to the variance. As it turns out, this is a necessity,
as illustrated by the following example.

\begin{example}Suppose that $X_{n}\sim\mathrm{Uniform}(\tilde{B}_{p,n})$
for $p$ an even natural number, and consider the following polynomial
$f_{n}(x)=\frac{1}{\sqrt{n}}\left(\|x\|_{p}^{p}-\EE\left[\|X_{n}\|_{p}^{p}\right]\right)$
of degree $p$. Then $\mathrm{coeff}_{p}(f)=1$. However, an easy
calculation (see Section \ref{subsec:Further-discussion}) shows,
\begin{equation}
\mathrm{Var}\left(\frac{1}{\sqrt{n}}\|X_{n}\|_{p}^{p}\right)=\EE\left[f_{n}^{2}(X_{n})\right]\xrightarrow{n\to\infty}0.\label{eq:badexample}
\end{equation}
\end{example}

One may wonder whether polynomials satisfying \eqref{eq:badexample}
are abundant, or whether it is some pathological example. When $X\sim\mathrm{Uniform}(\tilde{B}_{2,n})$,
the following proposition shows that the latter holds, i.e. the polynomial
$\frac{1}{\sqrt{n}}\|x\|_{2}^{2}$ is essentially the only bad example.
We do this by providing a complete description of the eigenvalues
and eigenvectors of the matrix $\mathrm{Cov}\left(X^{\otimes d}\right)$
in terms of spherical harmonics. For a more complete picture we refer
to Section \ref{sec:spectrum}.

\begin{proposition}[see Corollary \ref{Cor:spectrum of C}]Let
$X_{n}\sim\mathrm{Uniform}(\tilde{B}_{2,n})$. Write $\lambda_{1}\leq\lambda_{2}\leq\dots$
for the eigenvalues of the matrix $\mathrm{Cov}\left(X^{\otimes d}\right)$,
in increasing order. Then the following hold: 
\begin{enumerate}
\item ``\emph{pathological spectral gap}'': If $d=2$, then 
\[
\lambda_{1}=\frac{4}{n+4}=O(n^{-1}),
\]
has multiplicity one, with eigenvector $\frac{1}{\sqrt{n}}\|x\|_{2}^{2}$,
and the rest of the eigenvalues are bounded from below by $\frac{5}{7}$. 
\item For $d\geq3$ we have a uniform lower bound on the eigenvalues 
\[
\lambda_{i}\geq\frac{1}{(d+1)!},
\]
for all $n$. If $n\geq d$, then the $\lambda_{1}$-eigenspace is
spanned by monomials of the form $x_{i_{1}}\dots x_{i_{d}}$ with
$i_{1}<\dots<i_{d}$. 
\end{enumerate}
\end{proposition}

We remark that the lower bound in Item (2) can be further improved
(see Remark \eqref{rem:potential improvement}), and in fact $\underset{n\rightarrow\infty}{\lim}\lambda_{1}=1$
whenever $d\geq3$. 

\paragraph{Sub-level set estimates:}

Anti-concentration of polynomials with log-concave variables is a
well studied topic with many known results, most notably the work
of Carbery and Wright (\cite{carbery2001distributional}, but see
also \cite{nazarov2002geometric}), which also established reverse
H{\"o}lder inequalities. However, the results listed above are, in some
sense, of a different flavor.

In brief, (see Theorem \ref{thm:CW} below for exact formulation),
the Carbery-Wright inequality says that if $X$ is log-concave and
$f$ is a degree $d$ polynomial, then for every $\eps>0$, 
\[
\PP\left(\left|f(X)\right|\leq\eps\right)\lesssim\frac{\eps^{\frac{1}{d}}}{\EE\left[\left|f(X)\right|^{2}\right]^{\frac{1}{2d}}}.
\]
In other words, the inequality says something about sublevel sets
of the form $\{x\in\RR^{n}:\left|f(x)\right|\leq\eps\}$ under a moment
assumption.

In the same context, our result can roughly be stated as: if the coefficients
of $f$ are large, then $f(x)$ is not too concentrated around its
mean, in the sense that the variance is large.\\

While the results are not implied by nor imply one another, they turn
out to be complementary. By combining our results we then obtain the
following corollary, which is essentially a sublevel estimate, where
the moment assumption is replaced by an assumption on the coefficients.
\begin{corollary} \label{cor:smallball} Let $\mu$ be a log-concave
measure on $\RR^{n}$ with $X\sim\mu$ and let $f:\RR^{n}\to\RR$
be a polynomial of degree $d$. Fix $\eps>0$, 
\begin{enumerate}
\item If $\mu=\nu^{\otimes n}$ is an isotropic, log-concave product measure,
then there exists a universal constant $C>0$, such that for any $y\in\RR$,
\[
\PP\left(\left|f(X)-y\right|\leq\eps\right)\leq Cd\left(\frac{\eps}{\mathrm{coeff}_{d}(f)}\right)^{\frac{1}{d}}.
\]
\item If $\mu=\mathrm{Uniform}(\tilde{B}_{p,n})$, for some $p\geq1$ and
$f$ is homogeneous, 
\[
\PP\left(\left|f(X)\right|\leq\eps\right)\leq C_{d}\left(\frac{\eps}{\mathrm{coeff}_{d}(f)}\right)^{\frac{1}{d}},
\]
where $C_{d}>0$ is a constant which depends only on $d$. 
\end{enumerate}
\end{corollary} Note that in Item (2), we provide estimates only
for balls around $0$. Similarly as in the discussion after Theorem
\ref{thm:pballs}, by taking $f_{n}(x)=\frac{1}{\sqrt{n}}\|x\|_{p}^{p}$
and $X_{n}\sim\mathrm{Uniform}(\tilde{B}_{p,n})$ one can see there
is no hope for uniform estimates for sets of the form $\{\left|f(X)-y\right|\leq\eps\}$.

There are some other works which considered anti-concentration of
polynomials under an assumption on the coefficients (\cite{costello2006random,costello2016bilinear,razborov2013real,meka2016anti}),
mostly as part of the Littlewood-Offord theory, which first introduced
the problem for linear maps. However, previous results were constrained
to multi-linear polynomial with some combinatorial properties. 
Another related paper is \cite{emschwiller2020neural}, where dimension-dependent results were obtained in a similar setting to the one considered here.
We
also mention the work of Paouris (\cite{paouris2012small}, see also
\cite{klartag2007small}), which derived a similar result for the
push-forward of general log-concave measures under linear maps.

\subsection{\label{subsec:Decay-of-Fourier}Decay of Fourier coefficients}

Given a measure $\mu$ on $\RR^{n}$ and a polynomial $f:\RR^{n}\to\RR$,
anti-concentration results (as in Corollary \ref{cor:smallball})
can be rephrased by saying that the density of the pushforward measure
$f_{*}\mu$ does not explode too quickly around its singular values.
This explosion rate is controlled by the rate of decay of the Fourier
coefficients $t\mapsto\mathcal{F}(f_{*}\mu)(t)$ of $f_{*}\mu$. In
fact, using standard Fourier analysis, one can show that in order
to prove anti-concentration inequalities as in Corollary \ref{cor:smallball},
it is enough to give an upper bound of the form $\left|\mathcal{F}(f_{*}\mu)(t)\right|<C_{d}\cdot\left|t\right|^{-\frac{1}{d}}$
(for $\mathrm{coeff}_{d}(f)=1$).

In this work, we take the other direction and use our anti-concentration
results to obtain improved bounds on the decay of Fourier coefficients.
This reasoning is not new. Indeed, in the case $n=1$, one of the
earliest results, dating back to the 1920's, is the classical van
der Corput lemma connecting between derivatives of a function $f$
to the decay of its Fourier coefficients.

\begin{lemma}[{\cite[Proposition 2]{stein1993harmonic}}]\label{lem:1dvdc}Let
$f$ be a smooth function on $\RR$, and let $\mu=\rho(x)dx$ be a
measure on $\RR$. Fix $a<b$ and suppose that $k\in\NN$ is such
that $f^{(k)}\geq1$ for every $x\in(a,b)$. Then, if either of the
following conditions holds, 
\begin{itemize}
\item $k\geq2$, 
\item $k=1$, and $f'$ is monotonic, 
\end{itemize}
there exists a constant $C_{k}$, which depends only on $k$, such
that, 
\begin{equation}
\left|\mathcal{F}(f_{*}\mu|_{(a,b)})(t)\right|:=\left|\int\limits _{a}^{b}e^{\mathrm{i}tf(x)}d\mu(x)\right|\leq C_{k}\left|t\right|^{-\frac{1}{k}}\left(\rho(b)+\int\limits _{a}^{b}\left|\rho'(x)\right|dx\right).\label{eq:classic van der corput}
\end{equation}
\end{lemma} In \cite{CCW99}, a multivariate analogue of the van
der Corput lemma was obtained. In particular, given a degree $d$ polynomial $f:\RR^{n}\rightarrow\RR$, and $\mu=\mathrm{Uniform}([-1,1]^{n})$, if $\partial^{I}f|_{[-1,1]^{n}}\geq1$, where $\partial^{I}=\partial x_{1}^{I_{1}}...\partial x_{n}^{I_{n}}$,
for some $|I|=d$, then,
\begin{equation}
\left|\mathcal{F}(f_{*}\mu)(t)\right|<C_{d,n}\left|t\right|^{-\frac{1}{d}},\label{eq:van der Corput cube}
\end{equation}
where $C_{d,n}$ depends on $n$ and $d$ (see \cite[Theorem 7.2]{CCW99}).

Other than that, there have been many works on generalizing the van
der Corput lemma to higher dimensions, and by now there are plenty
of results for different classes of functions and domains (see for
example \cite{ruzhansky2012multidimensional,carbery2002what,gressman2016maximal,christ2005multilinear}).
However, to the best of our knowledge, none of these results include
dimension-free estimates. In \cite{carbery2001distributional}, Carbery
and Wright asked whether the constant in \eqref{eq:van der Corput cube}
can be replaced by a dimension-free constant, while only assuming
$\left\Vert f\right\Vert _{1}\geq C_{d}$ and $\int_{[-1,1]^{n}}f=0$.
Since our Theorem \ref{thm:product} is inherently dimension-free
we are able to prove the first dimension-free bound, which applies
to a large class of measures. In particular, when specializing to the cube, this gives a partial
answer to their question.

\begin{theorem}\label{thm:VDC} Let $\nu^{\otimes n}$ be an isotropic
log-concave product measure on $\RR^{n}$ and let $f:\RR^{n}\to\RR$,
be a polynomial of degree $d$, $f(x)=\sum\limits _{\left|I\right|\leq d}\alpha_{I}x^{I}$.
Denote $M_{d}(f)=\max\{\left|\alpha_{I}\right|:\left|I\right|=d\}$.
Then, for every $t \in \RR$,
\[
\left|\int\limits _{\RR^{n}}e^{\mathrm{i}tf(x)}d\nu^{\otimes n}(x)\right|\leq \frac{Cd}{(M_{d}(f)\left|t\right|)^{\frac{1}{d}}}.
\]
for some universal constant $C>0$. \end{theorem} Theorem \ref{thm:VDC}
gives a positive answer to the question posed in \cite{carbery2001distributional},
under the assumption that $M_{d}(f)\geq C'_{d}$. This is a stronger
requirement than the one in Theorem \ref{thm:product} which only
requires that $\mathrm{coeff}_{d}(f)\geq C'_{d}$, and both are stronger
than the condition than $\left\Vert f\right\Vert _{1}\geq C_{d}$
(by Theorem \ref{thm:product}). We do not know whether this is actually
necessary but point out that a recent analogous result for the Gaussian
measures obtained precisely the same dependence on the quantity $M_{d}(f)$
(see \cite[Corrolary 4.1]{kosov2020distributions}).

\subsection{Further discussion and future directions}

In this paper we study the pushforward $f_{*}\mu$ of a well-behaved
measure $\mu$ under polynomial maps $f:\RR^{n}\rightarrow\RR$ of
bounded degree. We focus on the regime where $n$ is arbitrarily large,
motivated by questions in high-dimensional geometry. There are a few
interesting variants which are worth mentioning. For simplicity of
presentation, we assume $f$ is a homogeneous polynomial of degree
$d$, but the discussion below easily generalizes to all polynomials.

First of all, one can also consider regimes of bounded complexity
(i.e. $n,d$ and $\mu$ are fixed), and try to obtain more refined
estimates than the ones afforded in the asymptotic realm. When $f$
and $\mu$ are fixed, it is known (see \cite{Igu78}, as well as \cite[Parts II,III]{AGV88}
and the references within) that the explosion rate of $\frac{df_{*}\mu}{dx}$
and the decay of Fourier coefficients of $f_{*}\mu$ are both controlled
by the singularities of $f$. The behavior of these singularities
can be quantified by the so-called \emph{log-canonical threshold}
of $f$, or $\mathrm{lct}(f)$, so that bad singularities correspond
to low values of the $\mathrm{lct}$ (see \cite{Mus12,Kol} for a
definition and an overview on the log-canonical threshold). In this
case, when $\varepsilon\ll1$, the term $\varepsilon^{\frac{1}{d}}$ in
Corollary \ref{cor:smallball} may actually be replaced by $\varepsilon^{\mathrm{lct}(f)}$.
Moreover, one always has $\mathrm{lct}(f)\geq1/d$. This suggests
that, while being tight, the Carbery-Wright inequality is somewhat
pessimistic, and could be improved for specific polynomial mappings.

With this in hand, it is still a non-trivial task to obtain effective
sublevel and Fourier estimates in terms of the $\mathrm{lct}$, which
are uniform on reasonable complexity classes of $\mu$, and with $\mathrm{deg}(f)$
bounded. One can further consider the case of polynomial maps $f:\RR^{n}\rightarrow\RR^{m}$,
for $m>1$. Here, $\mathrm{lct}(f)$ still controls the explosion
rate of $f_{*}\mu$ but does not control $\mathcal{F}(f_{*}\mu)$
anymore. Concrete uniform bounds will be of interest.

Secondly, instead of working over $\RR$, one can work with any local
field $F$. For $p$-adic fields, the study of $f_{*}\mu$, for suitable
$\mu$, is of arithmetic nature; for example, one can take the collection
of normalized Haar measures $\mu_{\mathbb{Z}_{p}^{n}}$ on $\mathbb{Z}_{p}^{n}$
(the ring of $p$-adic integers), which can be thought of as a $p$-adic
analogue of $\tilde{B}_{2,n}$ or the normalized Gaussian. For simplicity,
assume that $f$ has coefficients in $\ZZ$. Then, for each $k\in\NN$,
we have 
\begin{equation}
\PP\left(\left|f(X)\right|_{p}\leq p^{-k}\right)=\frac{\#\left\{ a\in\left(\ZZ/p^{k}\ZZ\right)^{n}:f(a)=0\mod p^{k}\right\} }{p^{kn}}
\end{equation}
(here $|\cdot|_{p}$ stands for the $p$-adic absolute value). Thus,
sublevel set estimates translate in the $p$-adic world into estimates
on the number of solutions of congruences of $f$ modulo $p^{k}$.
This is a fundamental question in number theory which is strongly
related to Igusa's local Zeta function (see e.g. \cite{Igu74,Den91a,DL98,dSG00}).
For a fixed $f$, sharp sublevel set estimates can be given (see \cite{Igu74},
and the discussion after \cite[Corollary 2.9]{VZG08}); there exists
a constant $C_{f,p}>0$ such that for all $k\in\NN$, 
\[
\PP\left(\left|f(X)\right|_{p}\leq p^{-k}\right)<C_{f,p}\cdot k^{n-1}p^{-k\mathrm{lct}(f)}.
\]

In \cite[Corollary 2.9]{VZG08} and \cite[Theorem 8.18]{GHb}, sublevel
set estimates were given for polynomial maps $f:\mathbb{Q}_{p}^{n}\rightarrow\mathbb{Q}_{p}^{m}$
for $m\geq1$, and more recently, sharp and field independent estimates
were given in \cite[Theorem 4.12]{CGH}.

In a similar fashion, the study of $\mathcal{F}(f_{*}\mu)$, translates
in the $p$-adic world to the study of \emph{exponential sums,} which
goes back to Gauss. The Fourier coefficients are essentially of the
following form: 
\begin{equation}
\frac{1}{p^{kn}}\sum_{x\in(\ZZ/p^{k}\ZZ)^{n}}\mathrm{exp}\left(\frac{2\pi \mathrm{i}f(x)}{p^{k}}\right).\label{eq:exponential sums}
\end{equation}
Igusa showed \cite{Igu78} that (\ref{eq:exponential sums}) can be
bounded from above by $C_{f,p}\cdot k^{n-1}\cdot p^{-k\mathrm{lct}(f)}$,
and further conjectured that $C_{f,p}$ can be replaced by $C_{f}$.
This was recently proved in \cite[Theorem 1.5]{CMN19}. Moreover,
$p$-adic analogues of the van der Corput lemma were given in \cite{Rog05,Clu11}.

It will be interesting to find sublevel and Fourier estimates in the
$p$-adic case, when the complexity is unbounded. This, along with
the variants presented above will be studied in a sequel to this paper.

\section{Orthogonal polynomials}

For the rest of this section we fix a centered measure $\mu$ on $\RR$,
such that for every $d\in\NN$, $\int\limits _{\RR}x^{d}d\mu(x)<\infty$
and such that $\mu$ is supported on infinitely many points. We associate
to $\mu$ a sequence of orthonormal polynomials $\{p_{d}\}_{d=0}^{\infty}$
satisfying, 
\[
\langle p_{d},p_{d'}\rangle_{L^{2}(\mu)}:=\int\limits _{\RR}p_{d}(x)p_{d'}(x)d\mu(x)=\delta_{d,d'}.
\]
Such a sequence may be obtained by applying the Gram-Schmidt algorithm
to the set $\{1,x,x^{2},\dots\}$, with respect to the standard inner
product on $L^{2}(\mu)$. Remark that, by definition, if $f\in L^{2}(\mu)$
then 
\begin{equation}
f=\sum\limits _{d=0}^{\infty}\langle f,p_{d}\rangle_{L^{2}(\mu)}p_{d},\label{eq:fourierdeco}
\end{equation}
where the equality is to be understood in $L^{2}(\mu)$, and where
\[
\langle f,p_{d}\rangle_{L^{2}(\mu)}=\int\limits _{\RR}f(x)p_{d}(x)d\mu(x).
\]
Observe that $p_{0}\equiv1$ and that since $\mu$ is centered, $p_{1}(x)\propto x$.
Moreover, it is easy to see that for every $d\in\NN$, $p_{d}$ is
a polynomial of degree $d$. The reader is referred to \cite{szego1975polynomials}
for more details pertaining to orthogonal polynomials. We will mostly
be interested in the following simple representation which is a consequence
of the Gram-Schmidt process.

Inside the Hilbert space $L^{2}(\mu)$, for $k\in\NN$, define $Q_{k}:L^{2}(\mu)\to L^{2}(\mu)$
as the orthogonal projection onto the closed subspace $\mathrm{span}\{1,x,x^{2},\dots,x^{k}\}$.
Then, 
\begin{equation}
p_{d}=\frac{1}{c_{\mu,d}}\left(x^{d}-Q_{d-1}x^{d}\right),\label{eq:orthoprojection}
\end{equation}
where the constant $c_{\mu,d}:=\left(\EE_{\mu}\left[\left(x^{d}-Q_{d-1}x^{d}\right)^{2}\right]\right)^{\frac{1}{2}}$
ensures that $\EE_{\mu}[p_{d}^{2}]=1$. Note that $c_{\mu,d}>0$. Indeed, since
$\mu$ is not supported on a finite number of points, $x^{d}\notin\mathrm{span}\{1,x,\dots,x^{d-1}\}.$

We now show that monomials have tractable expansions with respect
to the above orthogonal polynomials. \begin{lemma}\label{lem:monomials}Fix
$d\in\NN$, 
\begin{enumerate}
\item For any $k>d$, $\langle p_{k}(x),x^{d}\rangle_{L^{2}(\mu)}=0$. 
\item $\langle p_{d}(x),x^{d}\rangle_{L^{2}(\mu)}=c_{\mu,d},$ where $c_{\mu,d}$
is as defined by \eqref{eq:orthoprojection}. 
\end{enumerate}
\end{lemma} \begin{proof} Item (1) is a direct consequence of the
Gram-Schmidt process. For Item (2), note that since $Q_{d-1}$ is
an orthogonal projection, we have: 
\begin{align*}
c_{\mu,d} & =\frac{1}{c_{\mu,d}}\EE_{\mu}\left[\left(x^{d}-Q_{d-1}x^{d}\right)^{2}\right]=\frac{1}{c_{\mu,d}}\left(\EE_{\mu}\left[x^{2d}\right]-\EE_{\mu}\left[\left(Q_{d-1}x^{d}\right)^{2}\right]\right)\\
 & =\frac{1}{c_{\mu,d}}\langle x^{d}-Q_{d-1}x^{d},x^{d}\rangle_{L^{2}(\mu)}=\langle p_{d}(x),x^{d}\rangle_{L^{2}(\mu)},
\end{align*}
which concludes the proof. \end{proof} 
We next bound the constant $c_{\mu,d}$ from below. We start with the case of the cube and then use it to prove a general bound for isotropic log-concave measures.
 \begin{lemma} \label{lem:interval} Suppose that
$\mu=\mathrm{Uniform}([-1,1])$. Then, 
\[
c_{\mu,d}=\langle x^{d},p_{d}(x)\rangle_{L^{2}(\mu)}\geq\frac{1}{2^{d}}.
\]
\end{lemma} \begin{proof} In this case, the sequence $p_{d}$ is
given by the Legendre polynomials, and we have the following representation
(see \cite[Chapter 4]{szego1975polynomials}): 
\[
p_{d}(x)=\frac{\sqrt{2d+1}}{2^{d}d!}\frac{\partial^{d}}{\partial x^{d}}(x^{2}-1)^{d}.
\]
A direct calculation involving a $d$-fold integration by parts (see
e.g. \cite[Section 15]{AW01}) gives, 
\[
\langle x^{d},p_{d}(x)\rangle_{L^{2}(\mu)}=\frac{1}{\sqrt{2d+1}}\frac{2^{d}(d!)^{2}}{(2d)!}\geq\frac{1}{2^{d}}. 
\]
\end{proof}

\begin{lemma} \label{lem:1dlogconcave} Let $\mu$ be a log-concave
	and isotropic measure on $\RR$. Then $c_{\mu,d}\geq\frac{1}{9}\frac{1}{18^{d}}$.
\end{lemma}
\begin{proof}
	Write $f:=x^{d}-Q_{d-1}x^{d}$ and $\mu=g(x)dx$. Since $\mu$ is
	log-concave and isotropic, it follows e.g. from \cite[Lemma 5.5 and Theorem 5.14]{lovasz2007geometry}
	that $g(x)\geq\frac{1}{16}$ for all $x\in[-\frac{1}{9},\frac{1}{9}]$.
	Hence, we get:
	\begin{equation}
	c_{\mu,d}^{2}=\EE_{\mu}\left[f^{2}\right]\geq\frac{1}{16}\int_{-\frac{1}{9}}^{\frac{1}{9}}f(x)^{2}dx=\frac{1}{72}\cdot\frac{1}{2}\int_{-1}^{1}\widetilde{f}(t)^{2}dt,\label{eq:reduction to cube}
	\end{equation}
	where $\widetilde{f}(t):=f(\frac{1}{9}t)$ with  $\mathrm{coeff}_{d}(\widetilde{f})=9^{-d}.$
	Let us write $h_d(x):= \frac{\sqrt{2d+1}}{2^{d}d!}\frac{\partial^{d}}{\partial x^{d}}(x^{2}-1)^{d},$ for the Legendre polynomial of degree $d$, as in the proof of Lemma \ref{lem:interval}. So, from \eqref{eq:fourierdeco},
	$$\frac{1}{2}\int_{-1}^{1}\widetilde{f}(t)^{2}dt \geq  \left(\frac{1}{2}\int_{-1}^1 \tilde{f}(x)h_d(x)dx\right)^2.$$
	By first applying Item (1) of Lemma \ref{lem:monomials} and then Lemma \ref{lem:interval}, we get,
	$$\left(\frac{1}{2}\int_{-1}^{1} \tilde{f}(x)h_d(x)dx\right)^2 = \left(\frac{1}{2}\int_{-1}^{1} \frac{1}{9^d}x^dh_d(x)dx\right)^2 \geq \frac{1}{9^{2d}}\frac{1}{4^d}.$$
	The claim follows.
\end{proof}
\section{Anti-concentration of polynomials} \label{sec:anticoncentration}

\subsection{Product measures - Proof of Theorem \ref{thm:product}}

We now consider $\RR^{n}$ equipped with a product measure $\mu^{\otimes n}$,
where $\mu$ is some measure on $\RR$. Suppose that $\{p_{d}\}_{d=0}^{\infty}$
is the sequence of orthonormal polynomials, with respect to $\mu$,
as constructed in \eqref{eq:orthoprojection}.

To find an orthogonal decomposition of $L^{2}(\mu^{\otimes n})$,
for a multi-index $I=(I_{1},\dots,I_{n})\in\NN^{n}$ we define the
multivariate polynomial, 
\[
p_{I}(x):=\prod_{i=1}^{n}p_{I_{i}}(x_{i}).
\]
Since $L^{2}(\mu^{\otimes n})=L^{2}(\mu)^{\otimes n}$ we have that
the set $\{p_{I}\}_{I\in\NN^{n}}$ is a complete orthonormal system
in $L^{2}(\mu^{\otimes n})$. Our next step is to show that for degree
$d$ polynomials, the inner product with $p_{I}$ depends only on
the coefficient of $x^{I}$, as long as $\left|I\right|=d$. \begin{lemma}
\label{lem:multinomials} Fix $d\in\NN$ and let $q(x)=\sum\limits _{i=1}^{d}\sum\limits _{\substack{I\in\NN^{n}\\
\left|I\right|=i
}
}\alpha_{I}x^{I}$ be a degree $d$ polynomial in $\RR^{n}$. Then, for any $J\in\NN^{n}$
with $\left|J\right|=d$, 
\[
\langle q(x),p_{J}(x)\rangle_{L^{2}(\mu^{\otimes n})}=\alpha_{J}\prod_{i=1}^{n}c_{\mu,J_{i}},
\]
where $c_{\mu,J_{i}}$ is as in \eqref{eq:orthoprojection}. \end{lemma}
\begin{proof} Clearly, we have, 
\[
\langle q(x),p_{J}(x)\rangle_{L^{2}(\mu^{\otimes n})}=\sum\limits _{i=1}^{d}\sum\limits _{\substack{I\in\NN^{n}\\
\left|I\right|=i
}
}\alpha_{I}\langle x^{I},p_{J}(x)\rangle_{L^{2}(\mu^{\otimes n})}.
\]
We first claim that if $I\neq J$ with $\left|I\right|\leq d$, then
$\langle x^{I},p_{J}(x)\rangle_{L^{2}(\mu^{\otimes n})}=0$. Indeed,
since $\left|J\right|=d$, necessarily, there exists some $j\in[n]$
such that $J_{j}>I_{j}$. We now use the product structure to write,
\[
\langle x^{I},p_{J}(x)\rangle_{L^{2}(\mu^{\otimes n})}=\prod_{i=1}^{n}\langle x_{i}^{I_{i}},p_{J_{i}}(x_{i})\rangle_{L^{2}(\mu)}=0.
\]
The second equality follows from Lemma \ref{lem:monomials} which
implies $\langle x_{j}^{I_{j}},p_{J_{j}}(x_{j})\rangle_{L^{2}(\mu)}=0$.
So, 
\[
\langle q(x),p_{J}(x)\rangle_{L^{2}(\mu^{\otimes n})}=\alpha_{J}\langle x^{J},p_{J}(x)\rangle_{L^{2}(\mu^{\otimes n})}=\alpha_{J}\prod_{i=1}^{n}\langle x_{i}^{J_{i}},p_{J_{i}}(x_{i})\rangle_{L^{2}(\mu)}=\alpha_{J}\prod_{i=1}^{n}c_{\mu,J_{i}}.
\]
where we have used Lemma \ref{lem:monomials} for the last equality.
\end{proof} We are now in a position to prove Theorem \ref{thm:product}.
\begin{proof}[Proof of Theorem \ref{thm:product}]
From \eqref{eq:fourierdeco}, we have, 
\begin{equation}
\mathrm{Var}_{\mu^{\otimes n}}(f)=\sum\limits _{\substack{I\in\NN^{n}\\
\left|I\right|\neq0
}
}\langle f(x),p_{I}(x)\rangle_{L^{2}(\mu^{\otimes n})}^{2}\geq\sum\limits _{\substack{I\in\NN^{n}\\
\left|I\right|=d
}
}\langle f(x),p_{I}(x)\rangle_{L^{2}(\mu^{\otimes n})}^{2}=\sum\limits _{\substack{I\in\NN^{n}\\
\left|I\right|=d
}
}\alpha_{I}^{2}\prod_{i=1}^{n}c^2_{\mu,I_{i}},\label{eq:product measures-lower bound on variance}
\end{equation}
where the second equality is Lemma \ref{lem:multinomials}. Since
$c_{\mu,I_{i}}>0$ there exists a constant $C_{\mu,d}$ such that
for any $I\in\NN^{n}$ with $\left|I\right|=d$, $\prod_{i=1}^{n}c_{\mu,I_{i}}^{2}\geq C_{\mu,d}.$
This concludes Item (1). \\
Item (2) is now a direct consequence
of \eqref{eq:product measures-lower bound on variance} and
Lemma \ref{lem:1dlogconcave}. Indeed, 
\[
\mathrm{Var}_{\mu^{\otimes n}}(f)\geq\sum\limits _{\substack{I\in\NN^{n}\\
\left|I\right|=d
}
}\alpha_{I}^{2}\prod_{i:I_{i}>0}\left(\frac{1}{9}\cdot\frac{1}{18^{I_{i}}}\right)^{2}\geq\sum\limits _{\substack{I\in\NN^{n}\\
\left|I\right|=d
}
}\alpha_{I}^{2}\left(\frac{1}{9^{d}}\cdot\frac{1}{18^{d}}\right)^{2}\geq\mathrm{coeff}_{d}^{2}(f)\cdot\frac{1}{2^{15d}}.
\]\qedhere
\end{proof}
\subsection{\label{subsec:A-sub-level-estimate}A sublevel estimate for log-concave
product measures}

The aim of this subsection is to show that, when specializing Theorem
\ref{thm:product} to the case of log-concave measures, we can translate
our variance estimates into estimates on small balls probabilities,
or sublevel estimates. This is essentially the first Item of Corollary
\ref{cor:smallball}.

Our main tool for this is the celebrated inequality of Carbery-Wright,
which we state in the form suited to our needs. 

\begin{theorem}{(\cite[Theorem 8]{carbery2001distributional})}
\label{thm:CW} Let $\mu$ be a log-concave measure on $\RR^{n}$
and let $f:\RR^{n}\to\RR$ be a polynomial of degree $d$. Then, if
$X\sim\mu$, for every $\eps>0$, 
\[
\PP\left(\left|f(X)\right|\leq\eps\right)\leq Cd\frac{\eps^{\frac{1}{d}}}{\EE\left[f(X)^{2}\right]^{\frac{1}{2d}}}.
\]
\end{theorem} Thus, the Carbery-Wright inequality says that an estimate
for the sublevel sets of a polynomial $f$ may be obtained by bounding
the second moment of $f$, which is precisely the content of Theorem
\ref{thm:product}. \begin{proof}[Proof of Items 1 in Corollary
\ref{cor:smallball}] Fix $y\in\RR$ and define the polynomial $f_{y}(x)=f(x)-y$.
It is clear that $\mathrm{coeff}_{d}(f_{y})=\mathrm{coeff}_{d}(f)$.
Combining this fact with Theorem \ref{thm:product}, we deduce, 
\[
\EE\left[f_{y}(X)^{2}\right]\geq\mathrm{Var}(f_{y}(X))\geq\frac{1}{2^{15d}}\mathrm{coeff}_{d}^{2}(f).
\]
Now, Theorem \ref{thm:CW} implies, 
\[
\PP\left(\left|f(X)-y\right|\leq\eps\right)=\PP\left(\left|f_{y}(X)\right|\leq\eps\right)\leq Cd\frac{\eps^{\frac{1}{d}}}{\EE\left[f_{y}(X)^{2}\right]^{\frac{1}{2d}}}\leq C'd\left(\frac{\eps}{\mathrm{coeff}_{d}(f)}\right)^{\frac{1}{d}},
\]
for some constant $C'>0$.\end{proof}

\subsection{Anti concentration on $L_{p}$ balls - Proof of Theorem \ref{thm:pballs}}

In this subsection we fix some $p\geq1$ and the measure $\mu$ on
$\RR^{n}$, with density $\frac{1}{\left(\frac{2}{p}\Gamma(\frac{1}{p})\right)^{n}}e^{-\|x\|_{p}^{p}}dx$,
where $\Gamma$ stands for the Gamma function. Observe that $\mu$
is a log-concave product measure. Recall that, 
\[
B_{p,n}=\{x\in\RR^{n}:\|x\|_{p}\leq1\},
\]
is the unit ball with respect to the norm $\|\cdot\|_{p}$ and that
if $Z=(Z_{1},\dots,Z_{n})\sim\mu$ and $U\sim\mathrm{Uniform}([0,1])$
is independent from $Z$, then 
\begin{equation}
X=U^{\frac{1}{n}}\frac{Z}{\|Z\|_{p}},\label{eq:sampling}
\end{equation}
is uniformly distributed on $B_{p,n}$ (see \cite{schechtman90volume}).
In other words, to generate $X$, one can first generate $Z$ and
normalize by $\|Z\|_{p}$ to obtain something which is distributed according to the normalized cone measure
on the boundary of $B_{p,n}$. To get a random vector uniformly distributed
on $B_{p,n}$ all that is left is to choose a random scale according
to $U^{\frac{1}{n}}$.

Before proceeding, we need the following technical lemma. \begin{lemma}
\label{lem:marginalmoments}Let $p\geq1$, and $Z,\mu$ as above. Then,
for any $k>-n$, 
\[
\EE\left[\|Z\|_p^k\right]=\frac{\Gamma\left(\frac{n+k}{p}\right)}{\Gamma\left(\frac{n}{p}\right)}.
\]
Moreover, if $n > k^2$ and $k  \geq 2$, then
$$\frac{1}{20}p^{\frac{k}{p}}n^{-\frac{k}{p}} \leq \EE\left[\frac{1}{\|Z\|_p^k}\right] = \frac{\Gamma\left(\frac{n-k}{p}\right)}{\Gamma\left(\frac{n}{p}\right)} \leq 25p^{\frac{k}{p}}n^{-\frac{k}{p}}.$$
\end{lemma} 

\begin{proof} Note that for any function $h:\RR_{\geq0}\to\RR$,
we have the change of coordinates formula: 
\begin{equation}
\int\limits _{\RR^{n}}h(\|x\|_{p})dx=\frac{2^{n}\Gamma\left(\frac{1}{p}\right)^{n}}{p^{n-1}\Gamma\left(\frac{n}{p}\right)}\int\limits _{0}^{\infty}r^{n-1}h(r)dr.\label{eq:ppolarintegration}
\end{equation}
The pre-factor can be verified by integrating against the density
of $\mu$ (see also \cite{barth2005probabilistic}). The identity in \eqref{eq:ppolarintegration}
immediately implies: 
\begin{align*}
\EE\left[\|Z\|_p^k\right]&=\int\limits _{\RR^{n}}\|x\|_{p}^{k}d\mu(x)  =\frac{1}{\left(\frac{2}{p}\Gamma(\frac{1}{p})\right)^{n}}\int\limits _{\RR^{n}}\|x\|_{p}^{k}e^{-\|x\|_{p}^{p}}dx=\frac{p}{\Gamma\left(\frac{n}{p}\right)}\int\limits _{0}^{\infty}r^{n+k-1}e^{-r^{p}}dr\\
 & =\frac{1}{\Gamma\left(\frac{n}{p}\right)}\int\limits _{0}^{\infty}t^{\frac{n+k}{p}-1}e^{-t}dt=\frac{\Gamma\left(\frac{n+k}{p}\right)}{\Gamma\left(\frac{n}{p}\right)}.
\end{align*}
Now, suppose that $k \geq 2$ and $n > k^2$. To estimate $\EE\left[\frac{1}{\|Z\|_p^k}\right]$, we first consider the case $n<p$. For this, we use Wendel's inequality
for ratios of Gamma functions \cite{jameson2013inequalities}, to
deduce,
\begin{equation*}
p^{\frac{k}{p}}n^{-\frac{k}{p}}\leq\left(\frac{n-k}{p}\right)^{-\frac{k}{p}}\leq\frac{\Gamma\left(\frac{n-k}{p}\right)}{\Gamma\left(\frac{n}{p}\right)}\leq\frac{p}{n-k}\cdot\left(\frac{n}{p}\right)^{1-\frac{k}{p}}\leq2p^{\frac{k}{p}}n^{-\frac{k}{p}}.
\end{equation*}
When $n\geq p$, we use Stirling's approximation
for the Gamma function (\cite{jameson2015simple}), and the inequality
$\left(1-\frac{1}{x}\right)^{x}<\frac{1}{e}<\left(1-\frac{1}{x}\right)^{x-1}$
for $x>1$, to deduce:
\begin{align*}
\frac{\Gamma\left(\frac{n-k}{p}\right)}{\Gamma\left(\frac{n}{p}\right)} & \leq\frac{\left(\frac{n-k}{p}\right)^{-\frac{1}{2}}\left(\frac{n-k}{pe}\right)^{\frac{n-k}{p}}e^{\frac{p}{12(n-k)}}}{\left(\frac{n}{p}\right)^{-\frac{1}{2}}\cdot\left(\frac{n}{pe}\right)^{\frac{n}{p}}}\leq2e^{\frac{1}{6}}\frac{\left(1-\frac{k}{n}\right)^{\frac{n}{p}}}{\left(\frac{n-k}{p}\right)^{\frac{k}{p}}}e^{\frac{k}{p}}\nonumber \\
& \leq4\frac{p^{\frac{k}{p}}}{(1-\frac{k}{n})^{\frac{k}{p}}}n^{-\frac{k}{p}}\leq4\frac{p^{\frac{k}{p}}}{(1-\frac{1}{k})^{\frac{k}{p}}}n^{-\frac{k}{p}}\leq 4(2e)^{\frac{1}{p}}p^{\frac{k}{p}}n^{-\frac{k}{p}}\leq 25p^{\frac{k}{p}}n^{-\frac{k}{p}}.
\end{align*}
To get a corresponding bound in the other direction,
we similarly have: 
\begin{equation*}
\frac{\Gamma\left(\frac{n-k}{p}\right)}{\Gamma\left(\frac{n}{p}\right)}\geq e^{-\frac{p}{12n}}\frac{\left(1-\frac{k}{n}\right)^{\frac{n}{p}}}{\left(\frac{n-k}{p}\right)^{\frac{k}{p}}}e^{\frac{k}{p}}\geq e^{-\frac{1}{12}}\cdot\frac{\left(1-\frac{k}{n}\right)^{\frac{k}{p}}}{\left(\frac{n}{p}\right)^{\frac{k}{p}}}\geq\frac{p^{\frac{k}{p}}}{2}\left(1-\frac{1}{k}\right)^{\frac{k}{p}}n^{-\frac{k}{p}}\geq\frac{1}{20}p^{\frac{k}{p}}n^{-\frac{k}{p}}.
\end{equation*}
Combining the above displays finishes the proof.
\end{proof} We now prove the main result of this section, a lower
bound for the second moment of a homogeneous polynomial over the unit
$L_{p}$ ball. The main theorem will follow by appropriately re-scaling
$B_{p,n}$ to be isotropic. \begin{lemma} \label{lem:unnormalizedanti}
Let the above notations prevail and let $f:\RR^{n}\to\RR$ be a homogeneous
polynomial of degree $d$. Then as long as $n>16d^{2}$, 
\[
\EE[f^{2}(X)]\geq C_{d}\cdot n^{-\frac{2d}{p}}\mathrm{coeff}_{d}^{2}(f).
\]
\end{lemma} \begin{proof} With the above notations, let us estimate
\begin{align*}
\EE\left[f^{2}(X)\right]=\EE\left[f^{2}\left(U^{\frac{1}{n}}\frac{Z}{\|Z\|_{p}}\right)\right] & =\EE\left[\frac{U^{\frac{2d}{n}}}{\|Z\|_{p}^{2d}}f^{2}(Z)\right]\\
 & =\EE\left[U^{\frac{2d}{n}}\right]\EE\left[\frac{1}{\|Z\|_{p}^{2d}}f^{2}(Z)\right]=\frac{1}{\frac{2d}{n}+1}\EE\left[\frac{1}{\|Z\|_{p}^{2d}}f^{2}(Z)\right].
\end{align*}
The first equality is \eqref{eq:sampling}, the second is the homogeneity
of $f$ and the third follows by independence of $U$. Fix $\delta>0$
and define the set, 
\[
A_{\delta}=\{x\in\RR^{n}:f^{2}(x)>\delta\}.
\]
So, we have 
\begin{equation}
\EE\left[\frac{1}{\|Z\|_{p}^{2d}}f^{2}(Z)\right]\geq\delta\EE\left[\frac{1}{\|Z\|_{p}^{2d}}{\bf 1}_{\{Z\in A_{\delta}\}}\right].\label{eq:bigevent}
\end{equation}
Moreover, by Cauchy-Schwartz, 
\[
\EE\left[\frac{1}{\|Z\|_{p}^{2d}}{\bf 1}_{\{Z\notin A_{\delta}\}}\right]\leq\sqrt{\EE\left[\frac{1}{\|Z\|_{p}^{4d}}\right]\PP\left(Z\notin A_{\delta}\right)}.
\]
Since we have assumed $n > 16d^2$, we can apply the second part of Lemma \ref{lem:marginalmoments} twice, for $k = 2d$ and $k = 4d$. Thus,
\begin{align}
\EE\left[f^{2}(X)\right] & \geq\frac{\delta}{\frac{2d}{n}+1}\EE\left[\frac{1}{\|Z\|_{p}^{2d}}{\bf 1}_{\{Z\in A_{\delta}\}}\right]\nonumber \\
 & =\frac{\delta}{\frac{2d}{n}+1}\left(\EE\left[\frac{1}{\|Z\|_{p}^{2d}}\right]-\EE\left[\frac{1}{\|Z\|_{p}^{2d}}{\bf 1}_{\{Z\notin A_{\delta}\}}\right]\right)\nonumber \\
 &\geq \frac{\delta}{\frac{2d}{n}+1}\left(\EE\left[\frac{1}{\|Z\|_{p}^{2d}}\right]-\sqrt{\EE\left[\frac{1}{\|Z\|_{p}^{4d}}\right]\PP\left(Z\notin A_{\delta}\right)}\right)\nonumber \\
 & \geq\frac{\delta p^{\frac{2d}{p}}}{20(\frac{2d}{n}+1)n^{\frac{2d}{p}}}\left(1-100\sqrt{\PP\left(Z\notin A_{\delta}\right)}\right).\label{eq:momentDecom}
\end{align}
We turn to estimate $\PP\left(Z\notin A_{\delta}\right)$. Applying
Lemma \ref{lem:marginalmoments} for a single coordinate, with $k=2$,
shows $\EE\left[Z_{1}^{2}\right]=\frac{\Gamma\left(\frac{3}{p}\right)}{\Gamma\left(\frac{1}{p}\right)}\geq\frac{1}{4}$,
where the inequality follows from Wendel's inequality, \cite{jameson2013inequalities}. Since $Z$
is also log-concave, we may invoke Item 1 of Corollary \ref{cor:smallball}.
So, 
\[
\PP\left(Z\notin A_{\delta}\right)=\PP\left(f^{2}(Z)\leq\delta\right)=\PP\left(\left|f(Z)\right|\leq\sqrt{\delta}\right)\leq Cd\left(\frac{\delta}{\mathrm{coeff}_{d}^{2}(f)}\right)^{\frac{1}{2d}}.
\]
Let us choose 
\[
\delta=\frac{\mathrm{coeff}_{d}^{2}(f)}{\left(10^{5}Cd\right)^{2d}}
\]
and plug it into \eqref{eq:momentDecom}. As long as $2d<n$, we obtain,
\[
\EE[f^{2}(X)]\geq\frac{\delta p^{\frac{2d}{p}}}{80}n^{-\frac{2d}{p}}\geq\frac{\delta}{80}n^{-\frac{2d}{p}}.
\]\qedhere
\end{proof} Theorem \ref{thm:pballs} is now an immediate consequence. 

\begin{proof}[Proof of Theorem \ref{thm:pballs}] Let $X=(X_{1}\dots,X_{n})\sim\mathrm{Uniform}(B_{p,n})$
and let $z_{p,n}=\EE\left[X_{1}^{2}\right]^{-\frac{1}{2}}$ be such
that $z_{p,n}X$ is isotropic, that is, $z_{p,n}X\sim\mathrm{Uniform}(\tilde{B}_{p,n})$.
It follows e.g. from \cite[Theorem 7]{barth2005probabilistic}, that
$z_{p,n}\geq C\cdot n^{\frac{1}{p}}$, for an absolute constant $C>0$.
If $n>16d^{2}$ then our claim follows by Lemma \ref{lem:unnormalizedanti}
and homogeneity, 
\[
\EE\left[f^{2}(z_{p,n}X)\right]=z_{p,n}^{2d}\EE\left[f^{2}(X)\right]\geq C^{2d}n^{\frac{2d}{p}}C_d\frac{\mathrm{coeff}_{d}^{2}(f)}{n^{\frac{2d}{p}}}=C^{2d}C_d\mathrm{coeff}_{d}^{2}(f).
\]
When $n\leq16d^{2}$, we can use the fact that $\tilde{B}_{p,n}$
contains a cube of length uniformly bounded from below by a constant
depending on $d$. Our claim then follows from Theorem \ref{thm:product}.
The proof is complete. \end{proof} We may now also prove Item (2)
of Corollary \ref{cor:smallball}. \begin{proof}[Proof of Item 2
in Corollary \ref{cor:smallball}] The proof is essentially identical
to the case of product measures. If $X\sim \mathrm{Uniform}(\tilde{B}_{p,n})$, from Theorem \ref{thm:CW},
\[
\PP\left(\left|f(X)\right|\leq\eps\right)\leq Cd\frac{\eps^{\frac{1}{d}}}{\EE\left[f(X)^{2}\right]^{\frac{1}{2d}}}\leq C_{d}\left(\frac{\eps}{\mathrm{coeff}_{d}(f)}\right)^{\frac{1}{d}},
\]
where the second inequality is Theorem \ref{thm:pballs}. \end{proof}

\section{Dimension-free van der Corput estimates}

Fix a measure $\mu$ on $\RR^{n}$ and a function $f:\RR^{n}\to\RR$.
The aim of this section is to bound the following quantity from above:
\[
\left|\int\limits _{\RR^{n}}e^{\mathrm{i}tf(x)}d\mu(x)\right|.
\]
In other words, if $f_{*}\mu$ is the push-forward of the measure
$\mu$, we wish to study the rate of decay of the Fourier coefficients
of $f_{*}\mu$. We first prove a variant of Lemma \ref{lem:1dvdc} for polynomials and isotropic log-concave measures on the real line. 
\begin{lemma} \label{lem:logconcavevdc}Let $\mu$ be an isotropic
log-concave measure on $\RR$, $f:\RR^n \to \RR$ a polynomial of degree $d$ and
$k\in[1,d]$, an integer. Then, for every $t\in \RR$,
\[
\left|\int\limits _{\{x\in\RR:|f^{(k)}(x)|\geq1\}}e^{\mathrm{i}tf(x)}d\mu(x)\right|\leq C\cdot dk\left|t\right|^{-\frac{1}{k}},
\]
for some universal constant $C > 0$.
\end{lemma} \begin{proof} We start by observing that since $\mu$
is log-concave its density $\rho$ is uni-modal So, there exists
a point $x_{0}\in\RR$, such that $\rho$ is non-decreasing up to $x_{0}$
and non-increasing from $x_{0}$. This immediately implies $\int\limits _{a}^{b}\left|\rho'(x)\right|dx\leq2\sup\limits _{x\in\RR}\rho(x)$,
for every interval $[a,b]\subseteq\RR$. Furthermore, since $\mu$
is isotropic, by \cite[Lemma 5.5]{lovasz2007geometry}, $\sup\limits _{x\in\RR}\rho(x)\leq1$.\\

Let $\beta>0$ be a real number equal to $1$ if $k=1$, and to be
determined later for $k\geq2$, and define the sets, 
\[
A_{1}=\{x\in\RR:|f'(x)|\geq\beta\text{ and }|f^{(k)}(x)|\geq1\},
\]
\[
A_{2}=\{x\in\RR:|f'(x)|<\beta\text{ and }|f^{(k)}(x)|\geq1\}.
\]
We decompose the integral on these sets to obtain, 
\[
\left|\int\limits _{\{x\in\RR:|f^{(k)}(x)|\geq1\}}e^{\mathrm{i}tf(x)}d\mu(x)\right|\leq\left|\int\limits _{A_{1}}e^{\mathrm{i}tf(x)}d\mu(x)\right|+\left|\int\limits _{A_{2}}e^{\mathrm{i}tf(x)}d\mu(x)\right|.
\]
Since $f^{(k)}$ is a polynomial of degree less than $d$, its derivative may change signs at most $d$ times. So, $\{x\in\RR:|f^{(k)}(x)|\geq1\}$ can be decomposed as a union of $M$ pairwise disjoint intervals, with $M \leq d$.
Explicitly, we have the following identity,
\begin{equation}
\{x\in\RR:|f^{(k)}(x)|\geq1\}=\bigcup\limits _{i=1}^{M}[a_{i},b_{i}],\label{eq:partition into intervals}
\end{equation}
where on each interval either $f^{(k)}(x)\geq1$, or $f^{(k)}(x)\leq-1$.
For the region $A_{2}$, since $\sup\limits _{x\in\RR}\rho(x)\leq1$,
we get,
\[
\left|\int\limits _{A_{2}}e^{\mathrm{i}tf(x)}d\mu(x)\right|\leq\int\limits _{A_{2}}\rho(x)dx\leq\mathrm{Vol}(A_{2}).
\]
For each $1\leq i\leq M$, the set $A_{2}\cap[a_{i},b_{i}]$ is a
sublevel set of $f'$ restricted to the region $[a_{i},b_{i}]$. When $k=1$,
$\mathrm{Vol}(A_{2})=0$, by our choice of $\beta$, and we need only consider $A_1$.
If $k\geq2$, we invoke the sublevel estimate in \cite[Proposition 2.1]{CCW99}\footnote{Proposition 2.1 of \cite{CCW99} is stated for functions which are
defined on the entire real line, but the proof works for functions
defined on any interval (see Section 2 of \cite{CCW99}).}, on each interval $[a_{i},b_{i}]$ separately, and sum the corresponding
volumes to obtain, 
\begin{equation}
\mathrm{Vol}(A_{2})=\sum_{i=1}^{M}\mathrm{Vol}(A_{2}\cap[a_{i},b_{i}])\leq\sum_{i=1}^{M}Ck\beta^{\frac{1}{k-1}}\leq Cdk\beta^{\frac{1}{k-1}}.\label{eq:A2}
\end{equation}

To handle $A_1$, we use the fact that both $f'$ and $f''$ are polynomials of degree less than $d$. Thus, a similar reasoning to the one used before allows us to
refine the partition in \eqref{eq:partition into intervals} into no more than
$4d$ intervals, with the property that on each interval $f'$ is monotone, and
either $f'(x)\geq\beta$, or $f'(x)\leq-\beta$, or $|f'(x)|\leq\beta$.
In particular, we can write $A_{1}$ as a disjoint union of $M'\leq4d$
intervals taken from this refined partition 
\[
A_{1}=\bigcup\limits _{i=1}^{M'}[c_{i},d_{i}],
\]
such that on each interval $[c_{i},d_{i}]$, either $f^{(k)}(x)\geq1$,
or $f^{(k)}(x)\leq-1$, and moreover $f'$ is monotone. 

For each $1\leq i\leq M'$, we integrate by parts, and use the bounds
$\rho(x)\leq1$ and $|f'(x)|\geq\beta$, to obtain: 
\begin{align*}
\left|\int\limits _{c_{i}}^{d_{i}}e^{\mathrm{i}tf(x)}d\mu(x)\right| & =\left|\int\limits _{c_{i}}^{d_{i}}\frac{tf'(x)}{tf'(x)}e^{\mathrm{i}tf(x)}\rho(x)dx\right|\leq\left|\left(e^{\mathrm{i}tf(x)}\frac{\rho(x)}{tf'(x)}\right)\Big\vert_{c_{i}}^{d_{i}}\right|+\left|\int\limits _{c_{i}}^{d_{i}}e^{\mathrm{i}tf(x)}\left(\frac{\rho(x)}{tf'(x)}\right)'dx\right|\\
 & \leq\left|\frac{\rho(d_{i})}{tf'(d_{i})}\right|+\left|\frac{\rho(c_{i})}{tf'(c_{i})}\right|+\int\limits _{c_{i}}^{d_{i}}\rho(x)\left|\left(\frac{1}{tf'(x)}\right)'\right|dx+\int\limits _{c_{i}}^{d_{i}}|\rho'(x)|\frac{1}{|t||f'(x)|}dx\\
 & \leq\frac{2}{|t|\beta}+\frac{1}{|t|}\int\limits _{c_{i}}^{d_{i}}\left|\left(\frac{1}{f'(x)}\right)'\right|dx+\frac{1}{|t|\beta}\int\limits _{c_{i}}^{d_{i}}|\rho'(x)|dx\\
 & \leq\frac{4}{|t|\beta}+\frac{1}{|t|}\left|\int\limits _{c_{i}}^{d_{i}}\left(\frac{1}{f'(x)}\right)'dx\right|\leq\frac{4}{|t|\beta}+\frac{1}{|t|}\left(\frac{1}{|f'(d_{i})|}+\frac{1}{|f'(c_{i})|}\right)\leq\frac{6}{|t|\beta}.
\end{align*}
When moving between the third and fourth lines we have used the fact
that $f'(x)$ is monotone on $[c_{i},d_{i}]$. Summing over all intervals
$[c_{i},d_{i}]$, we get
\[
\left|\int\limits _{A_{1}}e^{\mathrm{i}tf(x)}d\mu(x)\right|\leq\sum\limits _{i=1}^{M'}\left|\int\limits _{c_{i}}^{d_{i}}e^{\mathrm{i}tf(x)}d\mu(x)\right|\leq M'\frac{6}{|t|\beta}\leq\frac{24d}{|t|\beta}.
\]
Coupling this with \eqref{eq:A2}, we obtain, 
\[
\left|\int\limits _{\{x\in\RR:|f^{(k)}(x)|\geq1\}}e^{\mathrm{i}tf(x)}d\mu(x)\right|\leq\frac{24d}{|t|\beta}+Cdk\beta^{\frac{1}{k-1}}\leq\frac{24dk}{|t|\beta}+Cdk\beta^{\frac{1}{k-1}}.
\]
To conclude the proof we take $\beta=\frac{1}{|t|^{\frac{k-1}{k}}}$.
\end{proof}
Our result for log-concave product measures is a consequence of
the one-dimensional estimate coupled with the anti-concentration
result, proven in Section \ref{sec:anticoncentration}.
\begin{proof}[Proof of Theorem \ref{thm:VDC}]
Let $\mu=\nu^{\otimes n}$ be an isotropic log-concave product measure
on $\RR^{n}$. For convenience we denote, 
\[
J(t):=\left|\int\limits _{\RR^{n}}e^{\mathrm{i}tf(x)}d\mu(x)\right|.
\]
Now, let $I\in\NN^{n}$ with $\left|I\right|=d$, be such that $M_{d}(f)=\alpha_{I}$
and fix some $\eps>0$, to be determined later. We write $I=(\tilde{I},I_{n})$,
where $\tilde{I}$ is a multi-index on $n-1$ indices. Without
loss of generality, we may assume that $I_{n}\geq1$. Define the
set, 
\[
A:=\left\{ x\in\RR^{n}:\left|\frac{\partial^{I_{n}}}{\partial x_{n}^{I_{n}}}f(x)\right|\geq\eps\right\} .
\]
If $\bar{A} := \RR^n \setminus A$, then, 
\begin{align}
J(t)\leq & \left|\int\limits _{A}e^{\mathrm{i}tf(x)}d\mu(x)\right|+\left|\int\limits _{\bar{A}}e^{\mathrm{i}tf(x)}d\mu(x)\right|.\label{eq:coruptdecomp}
\end{align}
We estimate each term separately. First, observe that $\frac{\partial^{I_{n}}}{\partial x_{n}^{I_{n}}}f$
is a polynomial of degree at most $d-I_{n}$ and, clearly $\mathrm{coeff}_{d-I_{n}}\left(\frac{\partial^{I_{n}}}{\partial x_{n}^{I_{n}}}f\right)\geq I_n!M_{d}(f)$.
Hence, by applying Corollary \ref{cor:smallball} to $\frac{\partial^{I_{n}}}{\partial x_{n}^{I_{n}}}f$,
one has, 
\begin{align}
\left|\int\limits _{\bar{A}}e^{\mathrm{i}tf(x)}d\mu(x)\right|\leq\PP\left(\left|\frac{\partial^{I_{n}}}{\partial x_{n}^{I_{n}}}f(X)\right|\leq\eps\right)\leq Cd\left(\frac{\eps}{I_n!M_{d}(f)}\right)^{\frac{1}{d-I_{n}}}.\label{eq:abar}
\end{align}
To deal with the first term in \eqref{eq:coruptdecomp}, write $x=(\tilde{x},x_{n})$,
and note that, 
\begin{align*}
\left|\int\limits _{A}e^{\mathrm{i}tf(x)}d\nu^{\otimes n}(x)\right| & \leq\int\limits _{\RR^{n-1}}\left|\int\limits _{-\infty}^{\infty}e^{\mathrm{i}tf(\tilde{x},x_{n})}{\bf 1}_{A}d\nu(x_{n})\right|d\nu^{\otimes n-1}(\tilde{x})\\
 & =\int\limits _{\RR^{n-1}}\left|\int\limits _{\{x_{n}:|f_{\tilde{x}}^{(I_{n})}(x_{n})|\geq\eps\}}e^{\mathrm{i}tf_{\tilde{x}}(x_{n})}d\nu(x_{n})\right|d\nu^{\otimes n-1}(\tilde{x}),
\end{align*}
where $f_{\tilde{x}}(x_{n}):=f(\tilde{x},x_{n})$.
We invoke Lemma \ref{lem:logconcavevdc} , with $k = I_n$, on the polynomial $\frac{1}{\eps}f_{\tilde{x}}$, which yields,
$$\left|\int\limits _{\{x_{n}:|f_{\tilde{x}}^{(I_{n})}(x_{n})|\geq\eps\}}e^{\mathrm{i}tf_{\tilde{x}}(x_{n})}d\nu(x_{n})\right|=\left|\int\limits _{\{x_{n}:|\frac{1}{\eps}f_{\tilde{x}}^{(I_{n})}(x_{n})|\geq1\}}e^{\mathrm{i}t\eps\frac{1}{\eps}f_{\tilde{x}}(x_{n})}d\nu(x_{n})\right|\leq\frac{C'dI_{n}}{(|t|\eps)^{\frac{1}{I_{n}}}}\leq\frac{eC'd}{(|t|\frac{\eps}{I_{n}!})^{\frac{1}{I_{n}}}},
$$
for some constant $C'>0$, where in the last inequality we have used
$k\leq e(k!)^{\frac{1}{k}}$.
We have thus established, 
\[
\left|J(t)\right|\leq Cd\left(\frac{\eps}{I_{n}!M_{d}(f)}\right)^{\frac{1}{d-I_{n}}}+\frac{eC'd}{\left(\left|t\right|\frac{\eps}{I_{n}!}\right)^{\frac{1}{I_{n}}}}.
\]
Choose $\eps=I_n!M_{d}(f)^{\frac{I_{n}}{d}}\cdot\frac{1}{\left|t\right|^{\frac{d-I_{n}}{d}}}$,
to get, 
\[
\left|J(t)\right|\leq\frac{(C+eC')d}{\left(M_{d}(f)\left|t\right|\right)^{\frac{1}{d}}},
\]
as required.\end{proof}

\section{Spectrum of the covariance matrix for tensor powers- the case of
the Euclidean ball\label{sec:spectrum}}

The goal of this section is to compute the spectrum of $\mathrm{Cov}(X^{\otimes d})$,
when $X\sim\mathrm{Uniform}(\tilde{B}_{2,n})$. Since $p=2$ in this
subsection, we simply write $B_{n}$ (resp. $\tilde{B}_{n}$) instead
of $B_{2,n}$ (resp. $\tilde{B}_{2,n}$), and set $\mu=\mathrm{Uniform}(\tilde{B}_{n})$.
We further denote by $R_{n}$ the radius of $\tilde{B}_{n}$ .

Recall that $X^{\otimes d}$ is a random vector in $\mathrm{Sym}_{d}(\RR^{n})$,
which we identify with $\mathcal{P}_{d}(\RR^{n})$, the space of all
real-valued homogeneous polynomials of degree $d$. Taking the inner
product $\langle\sum a_{I}x^{I},\sum b_{J}x^{J}\rangle:=\sum_{I}a_{I}b_{I}$,
with $\{x^{I}\}_{|I|=d}$ as an orthonormal basis, one can represent
$\mathrm{Cov}(X^{\otimes d})$ by the matrix $C:=\left\{ C_{I,J}\right\} _{\left|I\right|,\left|J\right|=d}$,
where 
\[
C_{I,J}=\EE_{\mu}\left[x^{I+J}\right]-\EE_{\mu}\left[x^{I}\right]\EE_{\mu}\left[x^{J}\right].
\]
As it turns out, it will be more convenient to work with a different
inner product, whose naturality will be apparent soon.

\begin{definition}[see \cite{BBEM90}]\label{Def:Bombieri inner product}
Let $f=\sum_{I}a_{I}x^{I}$ and $g=\sum_{J}b_{J}x^{J}$ be in $\mathcal{P}_{d}(\RR^{n})$.
Let $D_{f}$ be the partial differential operator $\sum_{I}a_{I}\partial^{I}$,
where $\partial^{I}=\partial x_{1}^{I_{1}}...\partial x_{n}^{I_{n}}$.
The \emph{Bombieri inner product} is defined as follows: 
\[
\langle f,g\rangle_{B}:=D_{f}(g)=\sum_{I}I!\cdot a_{I}b_{I},
\]
where $I!:=I_{1}!...I_{n}!$. We define the corresponding \emph{Bombieri
norm}: 
\[
\left\Vert f\right\Vert _{B}=\sqrt{\sum_{I}I!\cdot a_{I}^{2}}.
\]
\end{definition}
Let us record one important observation, which will be used later on. Given $f\in\mathcal{P}_{d-q}(\RR^{n})$,
$h\in\mathcal{P}_{q}(\RR^{n})$ and $g\in\mathcal{P}_{d}(\RR^{n})$,
we have the following identity (see e.g. \cite[Lemma 11]{BD95}),
\begin{equation}
\langle hf,g\rangle_{B}:=D_{hf}(g)=D_{f}(D_{h}(g))=\langle f,D_{h}(g)\rangle_{B}.\label{eq:Bombieri differential identity}
\end{equation}

To see the connection with the matrix $C$, write 
\[
\widetilde{C}:=\left\{ \widetilde{C}_{I,J}\right\} _{\left|I\right|,\left|J\right|=d},\text{ where }\widetilde{C}_{I,J}=\frac{\EE_{\mu}\left[x^{I+J}\right]-\EE_{\mu}\left[x^{I}\right]\EE\left[x^{J}\right]}{I!}.
\]
Then for every $f=\sum_{I}a_{I}x^{I}$ and $g=\sum_{J}b_{J}x^{J}$
in $\mathcal{P}_{d}(\RR^{n})$, one has 
\begin{align}
\langle\widetilde{C}f,g\rangle_{B} & =\left\langle \sum_{I}\left(\sum_{J}a_{J}\widetilde{C}_{I,J}\right)x^{I},\sum_{I}b_{I}x^{I}\right\rangle _{B}=\sum_{I,J}I!b_{I}a_{J}\widetilde{C}_{I,J}\nonumber \\
 & =\langle Cf,g\rangle=\langle f,g\rangle_{L^{2}(\mu)}-\langle f,1\rangle_{L^{2}(\mu)}\langle g,1\rangle_{L^{2}(\mu)}.\label{eq:Covariance matrix with Bombieri}
\end{align}
Note that $\widetilde{C}:=D\cdot C$, where $D$ is the diagonal matrix,
$D_{I,I}=\frac{1}{I!}$. Also, while $\tilde{C}$ is not symmetric,
it is self-adjoint with respect to the Bombieri inner product.

For an $N\times N$-matrix $M$ with non-negative eigenvalues, we
denote by $0\leq\lambda_{1}(M)\leq\dots\leq\lambda_{N}(M)$, its eigenvalues
in increasing order. The main result of this section is a complete
characterization of $\{\lambda_{i}(\tilde{C})\}$, along with their
corresponding eigenspaces (Theorem \ref{thm:spectrum of Ctilde}).
We then use the connection between $\tilde{C}$ and $C$, to deduce
information about the spectrum of $C$ (Corollary \ref{Cor:spectrum of C}).
Since the matrix $\tilde{C}$ depends on the parameters $n$ and $d$,
the same is also true for the quantities $\lambda_{i}(\tilde{C})$.
In the sequel, we suppress this dependence to simplify the notation.

\subsection{Spherical harmonics}

Before we state the main result, we need to collect a few basic facts
about spherical harmonics. We refer to \cite[Chapter 5]{ABR01} and
\cite[Chapter 2]{AG01} for more details.

We write $\mathcal{H}_{d}(\RR^{n})$ for the subspace of $\mathcal{P}_{d}(\RR^{n})$
consisting of all degree $d$ homogeneous harmonic polynomials on
$\RR^{n}$, and $\mathcal{H}_{d}(\mathbb{S}^{n})$ for its restriction
to the unit sphere $\mathbb{S}^{n}$. Let $\mu_{\mathbb{S}^{n}}$
be the unique $\mathrm{SO}_{n}(\RR)$-invariant probability measure
on the $n-1$-dimensional sphere $\mathbb{S}^{n}$. Denote by $L^{2}(\mathbb{S}^{n})$
the space of $L^{2}$-integrable real valued functions on the sphere,
with the inner product 
\[
\langle f,g\rangle_{L^{2}(\mathbb{S}^{n})}=\int_{\mathbb{S}^{n}}f\cdot gd\mu_{\mathbb{S}^{n}}.
\]

It turns out that the inner products $\langle\,,\,\rangle_{_{L^{2}(\mu)}}$,
$\langle\,,\,\rangle_{_{L^{2}(\mathbb{S}^{n})}}$ and $\langle\,,\,\rangle_{B}$
are comparable on the subspace of $d$-harmonic polynomials. \begin{lemma}\label{lem:integrals of harmonic polynomials}
Let $f,g\in\mathcal{H}_{d}(\RR^{n})$. Then we have: 
\[
\langle f,g\rangle_{_{L^{2}(\mathbb{S}^{n})}}=\gamma_{d,n}\cdot\langle f,g\rangle_{B}\text{ and }\langle f,g\rangle_{_{L^{2}(\mu)}}=\frac{n}{n+2d}R_{n}^{2d}\langle f,g\rangle_{_{L^{2}(\mathbb{S}^{n})}},
\]
\[
\gamma_{d,n}=\frac{1}{n(n+2)...(n+2d-2)}.
\]

\end{lemma} \begin{proof} The fact that $\langle f,g\rangle_{L^{2}(\mathbb{S}^{n})}=\gamma_{d,n}\cdot\langle f,g\rangle_{B}$
follows e.g. from \cite[Theorem 5.14]{ABR01}. For the second claim,
recall that the isotropic ball $\tilde{B}_{n}$ has radius $R_{n}$
and volume $V_{n}=R_{n}^{n}\cdot\mathrm{Vol}(B_{n})$. Writing $d\sigma$
for the surface measure on $\mathbb{S}^{n}$ (so that $\sigma=n\cdot\mathrm{Vol}(B_{n})\mu_{\mathbb{S}^{n}}$),
one has 
\begin{align*}
\langle f,g\rangle_{L^{2}(\mu)} & =\int_{\tilde{B}_{n}}f(x)g(x)d\mu(x)=\frac{1}{R_{n}^{n}\cdot\mathrm{Vol}(B_{n})}\int_{0}^{R_{n}}r^{2d+n-1}\left(\int_{\mathbb{S}^{n}}f(x)\cdot g(x)d\sigma(x)\right)dr\\
 & =\frac{n}{n+2d}R_{n}^{2d}\left(\int_{\mathbb{S}^{n}}f(x)\cdot g(x)d\mu_{\mathbb{S}^{n}}(x)\right)dr=\frac{n}{n+2d}R_{n}^{2d}\langle f,g\rangle_{L^{2}(\mathbb{S}^{n})}.
\end{align*}
\end{proof} \begin{lemma}[{see e.g. \cite[Theorem 5.12]{ABR01}
and \cite[Theorem 2.1.1]{AG01}}]\label{lem:structure theory of homogeneous polynomials}~ 
\begin{enumerate}
\item The Hilbert space $L^{2}(\mathbb{S}^{n})$ can be decomposed into
a direct sum $L^{2}(\mathbb{S}^{n})=\widehat{\bigoplus}_{l\in\NN}\mathcal{H}_{l}(\mathbb{S}^{n})$,
where $\mathcal{H}_{l}(\mathbb{S}^{n})$ is orthogonal to $\mathcal{H}_{m}(\mathbb{S}^{n})$
for every $m\neq l$. 
\item For each $d\geq2$, we have an $\langle\,,\,\rangle_{B}$-orthogonal
decomposition 
\begin{equation}
\mathcal{P}_{d}(\RR^{n})=\mathcal{H}_{d}(\RR^{n})\oplus\left\Vert x\right\Vert ^{2}\mathcal{H}_{d-2}(\RR^{n})\oplus...\oplus\left\Vert x\right\Vert ^{2\left\lfloor \frac{d}{2}\right\rfloor }\mathcal{H}_{d-2\left\lfloor \frac{d}{2}\right\rfloor }(\RR^{n}).\label{eq:orthogonal decomposition of the Bombieri norm}
\end{equation}
\end{enumerate}
\end{lemma}

\subsection{Calculation of the spectrum}
We are now ready to state the main theorem which describes the spectrum
of $\widetilde{C}$, and in fact shows that the decomposition in \eqref{eq:orthogonal decomposition of the Bombieri norm}
is an eigenspace decomposition. \begin{theorem} \label{thm:spectrum of Ctilde}
Each subspace $\left\Vert x\right\Vert ^{2i}\mathcal{H}_{d-2i}(\RR^{n})$
of $\mathcal{P}_{d}(\RR^{n})$, with $i\in\left\{ 0,\dots,\left\lfloor \frac{d}{2}\right\rfloor \right\} $,
is a $\widetilde{C}$-eigenspace with eigenvalue $\eta_{i}$, where
\[
\eta_{i}=R_{n}^{2d}\frac{n}{n+2d}\cdot\frac{1}{2^{i}i!n(n+2)...(n+2d-2i-2)}
\]
if $i<\frac{d}{2}$, and whenever $d$ is even, 
\[
\eta_{\frac{d}{2}}=R_{n}^{2d}\frac{d^{2}}{2^{\frac{d}{2}}\left(\frac{d}{2}\right)!\left(n+2d\right)\left(n+d\right)\prod\limits _{j=0}^{\frac{d}{2}-1}(d+n-2j)}.
\]
In particular, the multiplicity $\mathrm{mult}(\eta_{i})$ of the
eigenvalue $\eta_{i}$ is equal to the dimension of $\mathcal{H}_{d-2i}(\RR^{n})$:
\[
\mathrm{mult}(\eta_{i})=\left(\begin{array}{c}
n+d-2i-1\\
n-1
\end{array}\right)-\left(\begin{array}{c}
n+d-2i-3\\
n-1
\end{array}\right).
\]
\end{theorem} \begin{proof} Write $\mathcal{P}_{d}(\RR^{n})=\bigoplus_{i=0}^{\left\lfloor \frac{d}{2}\right\rfloor }W_{i}$,
where $W_{i}:=\left\Vert x\right\Vert ^{2i}\mathcal{H}_{d-2i}(\RR^{n})$.

\noindent First note that for every $f\in W_{i}$, $g\in W_{j}$ with
$i\neq j$, we have, 
\[
\langle\tilde{C}f,g\rangle_{B}=\langle f,g\rangle_{L^{2}(\mu)}=R_{n}^{2d}\frac{n}{n+2d}\langle f,g\rangle_{L^{2}(\mathbb{S}^{n})}=0.
\]
The first equality is \eqref{eq:Covariance matrix with Bombieri},
the second is Lemma \ref{lem:integrals of harmonic polynomials} and
the third follows from the first item of Lemma \ref{lem:structure theory of homogeneous polynomials}.
We see that $\widetilde{C}(W_{i})$ is orthogonal to $W_{j}$ for
all $j\neq i$ and therefore $\widetilde{C}(W_{i})=W_{i}$. Furthermore,
the same reasoning shows that if $f\in W_{0}$, then 
\[
\langle\widetilde{C}f,f\rangle_{B}=\langle f,f\rangle_{L^{2}(\mu)}-\langle f,1\rangle_{L^{2}(\mu)}^{2}=R_{n}^{2d}\frac{n}{n+2d}\gamma_{d,n}\langle f,f\rangle_{B}.
\]
Similarly, any $f\in W_{i}$ can be written as $f=\left\Vert x\right\Vert ^{2i}\cdot g(x)$
with $g(x)=\sum a_{I}x^{I}\in\mathcal{H}_{d-2i}(\RR^{n})$. Now, if
$\Delta$ stands for the Laplacian, it is straightforward to verify
(e.g. \cite[4.5 and 5.22]{ABR01}) that $\Delta^{\circ i}(\|x\|^{2i}g(x))=b_{i}g(x)$,
where 
\begin{equation}
b_{i}:=2^{i}i!\prod_{j=1}^{i}(n+2d-2j-2i).\label{eq:b_i}
\end{equation}
Moreover, by \eqref{eq:Bombieri differential identity},
$\Delta^{\circ i}$ is the conjugate of multiplication by $\|x\|^{2i}$
with respect to the Bombieri inner product. Thus, 
\begin{equation}
\langle f,f\rangle_{B}=\langle\left\Vert x\right\Vert ^{2i}g(x),\left\Vert x\right\Vert ^{2i}g(x)\rangle_{B}=\langle g(x),\Delta^{\circ i}(\left\Vert x\right\Vert ^{2i}g(x))\rangle_{B}=b_{i}\langle g,g\rangle_{B}.\label{eq:laplacian on harmonics}
\end{equation}
Using \eqref{eq:laplacian on harmonics} we obtain, for $i<\frac{d}{2}$:
\begin{align*}
\langle\widetilde{C}f,f\rangle_{B} & =\langle f,f\rangle_{L^{2}(\mu)}-\langle f,1\rangle_{L^{2}(\mu)}^{2}=R_{n}^{2d}\frac{n}{n+2d}\cdot\langle g,g\rangle_{L^{2}(\mathbb{S}^{n})}\\
 & =R_{n}^{2d}\frac{n}{n+2d}\gamma_{d-2i,n}\langle g,g\rangle_{B}\\
 & =R_{n}^{2d}\frac{n}{n+2d}\cdot\frac{1}{n(n+2)...(n+2d-4i-2)}\langle g,g\rangle_{B}\\
 & =R_{n}^{2d}\frac{n}{n+2d}\cdot\frac{1}{2^{i}i!n(n+2)...(n+2d-2i-2)}\langle f,f\rangle_{B}.
\end{align*}
Finally, for $2i=d$ (and $d$ even), using Lemma \ref{lem:integrals of harmonic polynomials}
we have 
\begin{align*}
\langle\widetilde{C}\left\Vert x\right\Vert ^{d},\left\Vert x\right\Vert ^{d}\rangle_{B} & =\langle\left\Vert x\right\Vert ^{d},\left\Vert x\right\Vert ^{d}\rangle_{L^{2}(\mu)}-\langle\left\Vert x\right\Vert ^{d},1\rangle_{L^{2}(\mu)}^{2}\\
 & =\frac{n}{n+2d}R_{n}^{2d}-\left(\frac{n}{n+d}R_{n}^{d}\right)^{2}\\
 & =R_{n}^{2d}\left(\frac{n}{n+2d}-\frac{n^{2}}{n^{2}+2dn+d^{2}}\right)\\
 & =R_{n}^{2d}\frac{d^{2}n}{\left(n+2d\right)\left(n+d\right)^{2}}.
\end{align*}
Note that 
\[
\langle\left\Vert x\right\Vert ^{d},\left\Vert x\right\Vert ^{d}\rangle_{B}=b_{d/2}=2^{d/2}\left(\frac{d}{2}\right)!\prod_{j=1}^{d/2}(d+n-2j),
\]
so 
\[
\langle\widetilde{C}\left\Vert x\right\Vert ^{d},\left\Vert x\right\Vert ^{d}\rangle_{B}=\frac{R_{n}^{2d}}{b_{d/2}}\frac{d^{2}n}{\left(n+2d\right)\left(n+d\right)^{2}}\langle\left\Vert x\right\Vert ^{d},\left\Vert x\right\Vert ^{d}\rangle_{B},
\]
as required.

To put everything together, we have shown that every $W_{i}$ is a
$\tilde{C}$-invariant subspace and that the Rayleigh quotient $\frac{\langle\tilde{C}f,f\rangle_{B}}{\langle f,f\rangle_{B}}$
is constant on $W_{i}$. Since $\tilde{C}$ is self-adjoint with respect
to $\langle\ ,\ \rangle_{B}$ we can conclude that it is a constant
multiple of the identity on $W_{i}$ and the claim follows. \end{proof}
Since $\tilde{B}_{n}$ is isotropic, it is well known that $R_{n}=\sqrt{n+2}$.
Let us now understand the quantities $\eta_{i}$ better. If $i<\frac{d}{2}$,
we have, 
\begin{align}
n^{i}\geq\eta_{i} & \geq\frac{n^{d}}{n+2d}\cdot\frac{1}{2^{i}i!(n+2)...(n+2d-2i-2)}=\frac{n^{i}}{(1+\frac{2d}{n})}\cdot\frac{1}{2^{i}i!(1+\frac{2}{n})...(1+\frac{2d-2i-2}{n})}\nonumber \\
 & \geq\frac{n^{i}}{d+1}\cdot\frac{1}{2^{i}i!(d-i)!}\geq\frac{n^{i}}{(d+1)!}.\label{eq:lower bound}
\end{align}
For the first inequality, we have used the definition of $\eta_i$, according to which, as long as $n \geq 2$,
$$
\eta_{i}=R_{n}^{2d}\frac{n}{n+2d}\cdot\frac{1}{2^{i}i!n(n+2)...(n+2d-2i-2)}\leq\frac{(n+2)^{d}}{2^{i}i!(n+2)^{d-i}}\leq\left(\frac{n}{2}+1\right)^{i}\leq n^{i}.
$$
 For the last inequality in \eqref{eq:lower bound}, we used the elementary estimate $\binom{d}{i} \geq \left(\frac{d}{i}\right)^i$, which implies $\frac{d!}{i!(d-i)!} = \binom{d}{i} \geq 2^{i}$, whenever $i < \frac{d}{2}$.
Combining \eqref{eq:lower bound} with a similar calculation for $i=\frac{d}{2}$,
when $d$ is even, one has: 
\begin{equation}
\eta_{\frac{d}{2}}=\Theta_{d}\left(n^{\frac{d}{2}-2}\right)\text{ and }\eta_{i}=\Theta_{d}\left(n^{i}\right),\label{eq:magnitude of eigenvalues}
\end{equation}
where $\Theta_{d}$ means we omit constants which depend only on
$d$.\\

Thus, when $d\geq3$, Theorem \ref{thm:spectrum of Ctilde} and the
discussion above give the following dimension-free bound: 
\begin{equation}
\lambda_{1}(\tilde{C})=\eta_{0}\geq\frac{1}{(d+1)!}.\label{eq:smallest eigen}
\end{equation}
If, on the other hand, $d=2$, then the smallest eigenvalue is 
\begin{equation}
\lambda_{1}(\tilde{C})=\eta_{1}=(n+2)^{2}\frac{4}{2\left(n+4\right)\left(n+2\right)^{2}}=\frac{2}{n+4},\label{eq:pathological}
\end{equation}
and is of multiplicity one, with eigenvector $\sum_{i=1}^{n}x_{i}^{2}$.
Indeed, we have 
\begin{equation}
\lambda_{2}(\tilde{C})=\eta_{0}=\frac{n+2}{n+4}>\lambda_{1}(\tilde{C})\label{eq:second eigen pathological case}
\end{equation}

Since $C$ is a product of $\widetilde{C}$ with the diagonal matrix
$D^{-1}$ (with $D_{I,I}^{-1}=I!$) we can now use the spectrum of
$\widetilde{C}$ to deduce information on the spectrum of $C$. \begin{corollary}\label{Cor:spectrum of C}Write
$\lambda_{1}(C)\leq\dots\leq\lambda_{N}(C)$ for the spectrum of $C$,
with $N=\left(\begin{array}{c}
d+n-1\\
d
\end{array}\right)$. Then, the following estimates hold: 
\begin{enumerate}
\item For all $i$, we have 
\[
d!\lambda_{i}(\widetilde{C})\geq\lambda_{i}(C)\geq\lambda_{i}(\widetilde{C}),
\]
where $\lambda_{1}(\widetilde{C})\leq\dots\leq\lambda_{N}(\widetilde{C})$
are explicitly given by Theorem \ref{thm:spectrum of Ctilde}. 
\item ``\emph{pathological spectral gap}'': If $d=2$, $n\geq3$ then
the smallest eigenvalue $\lambda_{1}(C)$ has multiplicity one, with
eigenvector $\sum_{i=1}^{n}x_{i}^{2}$, and 
\[
\lambda_{1}(C)=\frac{4}{n+4}=O(n^{-1}).
\]
The rest of the eigenvalues are bounded from below by $\frac{5}{7}$. 
\item For $d\geq3$ we have a uniform lower bound on the eigenvalues 
\begin{equation}
\lambda_{i}(C)\geq\frac{1}{(d+1)!},\label{eq:lower bounds on C-eigen}
\end{equation}
for all $n$. If $n\geq d$, then the $\lambda_{1}(C)$-eigenspace
is spanned by monomials of the form $x_{i_{1}}\dots x_{i_{d}}$ with
$i_{1}<\dots<i_{d}$. Moreover, we have 
\[
\underset{n\rightarrow\infty}{\lim}\lambda_{1}(C)=\underset{n\rightarrow\infty}{\lim}\lambda_{1}(\tilde{C})=1.
\]
\end{enumerate}
\end{corollary} \begin{proof} Note that $\tilde{C}=D\cdot C$ is
a product of two positive definite matrices. Since $D$ is a diagonal
matrix with diagonal entries in the range $[\frac{1}{d!},1]$, it
follows e.g. by \cite[Theorem 3]{WZ92} that: 
\[
d!\lambda_{i}(\tilde{C})\geq\lambda_{i}(C)\geq\lambda_{i}(\tilde{C}),
\]
which is Item (1). Item (2) now follows from \eqref{eq:pathological}
and \eqref{eq:second eigen pathological case}.\\
 If $n\geq d\geq3$, it is easy to verify that $\eta_{i+1}\geq\eta_{i}$
for all $0\leq i<\left\lfloor \frac{d}{2}\right\rfloor $. Item (1)
implies that $\lambda_{1}(C)\geq\lambda_{1}(\tilde{C})=\eta_{0}$.
On the other hand, monomials $x^{I}$ of the form $x_{i_{1}}\dots x_{i_{d}}$
with $i_{1}<\dots<i_{d}$ satisfy $\langle x^{I},x^{I}\rangle=\langle x^{I},x^{I}\rangle_{B}$
and they are harmonic, so 
\[
Cx^{I}=D^{-1}\widetilde{C}x^{I}=D^{-1}\eta_{0}x^{I}=\eta_{0}x^{I}.
\]
This shows that $\lambda_{1}(C)=\lambda_{1}(\tilde{C})=\eta_{0}$.
Finally, we have: 
\[
\underset{n\rightarrow\infty}{\lim}\lambda_{1}(C)=\underset{n\rightarrow\infty}{\lim}\lambda_{1}(\tilde{C})=\underset{n\rightarrow\infty}{\lim}\frac{(n+2)^{d}}{(n+2)...(n+2d-2)(n+2d)}=1
\]
which finishes Item (3). \end{proof}

\begin{remark}\label{rem:potential improvement} The lower bound
in \eqref{eq:lower bounds on C-eigen} can be improved, by considering
more refined estimates on the possible entries of $D$ and on $\eta_{i}$,
for small values on $n$. Since we already know that $\underset{n\rightarrow\infty}{\lim}\lambda_{1}(C)=\underset{n\rightarrow\infty}{\lim}\lambda_{1}(\tilde{C})=1$,
we chose to ignore this low-dimensional issue, and to keep the (slightly
non-optimal) current bound. \end{remark}

\subsubsection{\label{subsec:Further-discussion}Further discussion}

We conclude the section with a discussion on the asymptotic behavior
of spectrum of $C$ as well as on the general case of radial measures.

\subsubsection*{Partition of the spectrum into different asymptotic scales}

By combining Theorem \ref{thm:spectrum of Ctilde}, \eqref{eq:magnitude of eigenvalues}
and Item (1) of Corollary \ref{Cor:spectrum of C}, we see that the
eigenvalues $\lambda_{i}(C)$ can be partitioned into subsets $A_{0},...,A_{\left\lceil \frac{d}{2}-1\right\rceil }$
with respect to different asymptotic behaviors. The subset $A_{j}$
consists of eigenvalues which are of magnitude $\sim n^{j}$ (up to
a constants depending only on $d$). For $d=2$ there is an additional
eigenvalue $\lambda_{1}(C)\sim n^{-1}$ which belongs to a unique
asymptotic scale $A_{-1}$.

One may wonder whether this phenomenon can be generalized to other
families of measures with some form of symmetry. Namely, how general
is the situation where all eigenvalues of $\mathrm{Cov}(X^{\otimes d})$
converge to a discrete set of asymptotic scales as $n$ grows? In
particular, does it hold for the uniform measure on $L_{p}$ balls?

Let us consider the case when $X\sim\mathrm{Uniform}(\tilde{B}_{p,n})$
for $p$ an even natural number. Write $R_{n,p}$ for the radius of
$\tilde{B}_{p,n}$. Using a coordinate change, as in \eqref{eq:ppolarintegration},
it can be seen that the polynomial $f=\frac{1}{\sqrt{n}}\left\Vert x\right\Vert _{p}^{p}$
satisfies 
\begin{align}
\frac{\langle Cf,f\rangle}{\langle f,f\rangle} & =\mathrm{Var}\left(\frac{1}{\sqrt{n}}\|X\|_{p}^{p}\right)=\frac{1}{n}\left(\EE\left(\|X\|_{p}^{2p}\right)-\EE\left(\|X\|_{p}^{p}\right)^{2}\right)\nonumber \\
 & =\frac{R_{n,p}^{2p}}{n}\left(\frac{n}{n+2p}-\left(\frac{n}{n+p}\right)^{2}\right)\nonumber \\
 & =R_{n,p}^{2p}\frac{p^{2}}{(n+2p)(n+p)^{2}}=\Theta_{p}\left(n^{-1}\right).\label{eq:pathological Lp}
\end{align}
In particular, we see that the eigenvalues of $\mathrm{Cov}(X^{\otimes p})$,
are not bounded from below, in a way reminiscent of the Euclidean
case.

\subsubsection*{Radial measures}

The results of this section generalize to radial measures
of the form $\frac{d\mu}{dx}=\rho(\|x\|_{2})$, for some $\rho:\RR_{\geq0}\to\RR_{\geq0}$.
Indeed, the only difference lies at Lemma \ref{lem:integrals of harmonic polynomials},
where now we will have, 
\[
\langle f,g\rangle_{_{L^{2}(\mu)}}=\beta_{\mu,d}\langle f,g\rangle_{_{L^{2}(\mathbb{S}^{n})}},
\]
with, 
\[
\beta_{\mu,d}:=n\mathrm{Vol}(B_{n})\int\limits _{0}^{\infty}r^{n+2d-1}\rho(r)dr
\]
($\beta_{\mu,d}=\frac{n}{n+2d}R_{n}^{2d}$ in the case of the isotropic
Euclidean ball). In particular, the matrix $\widetilde{C}$ has the
same eigenspace decomposition as in Theorem \ref{thm:spectrum of Ctilde},
with eigenvalues $\eta_{\mu,i}=\eta_{i}\cdot\beta_{\mu,d}\cdot\frac{n+2d}{nR_{n}^{2d}}$,
 $i<\frac{d}{2}$, and $\eta_{\mu,\frac{d}{2}}=\frac{\beta_{\mu,d}-\beta_{\mu,d/2}^{2}}{b_{d/2}}$,
with $b_{d/2}$ as defined in \eqref{eq:b_i}. Consequently, Items
(1) and (3) of Corollary \ref{Cor:spectrum of C} hold for radial
measures as well, with slightly different lower bounds. Item (2),
i.e. the ``pathological spectral gap'' phenomenon, is true only
for certain classes of measures, and depends on $\beta_{\mu,d}.$
For example, for $\gamma_{n}$, the standard Gaussian in $\RR^{n}$,
a calculation shows, 
\[
\beta_{\gamma_{n},d}=\frac{n\mathrm{Vol}(B_{n})}{\sqrt{2\pi}^{n}}\int\limits _{0}^{\infty}r^{n+2d-1}e^{-r^{2}/2}dr=\frac{2^{d}n\mathrm{Vol}(B_{n})}{2\pi^{\frac{n}{2}}}\Gamma\left(\frac{n}{2}+d\right)=2^{d}\frac{\Gamma\left(\frac{n}{2}+d\right)}{\Gamma\left(\frac{n}{2}\right)}.
\]
Note that for $d=2$, we have 
\[
\eta_{\gamma_{n},1}=\frac{2}{n}\left(\frac{\Gamma\left(\frac{n}{2}+2\right)}{\Gamma\left(\frac{n}{2}\right)}-\frac{\Gamma\left(\frac{n}{2}+1\right)^{2}}{\Gamma\left(\frac{n}{2}\right)^{2}}\right)=\frac{1}{2n}\left(n(n+2)-n^{2}\right)=1.
\]
So, there is no pathological spectral gap. Thus, we can see that, in
contrast to the Euclidean ball, the spectrum $\eta_{\gamma_{n},i}$
is bounded uniformly from below, which is consistent with Theorem
\ref{thm:product}.

In fact, among all log-concave and isotropic radial measures, the
Euclidean ball is the extremal case, for which the pathological eigenvector
$\|x\|_{2}^{2}$ has the smallest eigenvalue. This is related to the
\emph{thin-shell} phenomenon, which states that every log-concave
and isotropic measure should be well concentrated around a Euclidean
sphere. A lower bound for thin-shell was proven in \cite[Theorem 2]{BK03},
where it was shown that $\mathrm{Var}\left(\frac{1}{\sqrt{n}}\|X\|_{2}^{2}\right)\geq\frac{4}{n+4}$
for every isotropic and log-concave $X$ in $\RR^{n}$, satisfying
a certain monotonicity assumption. As we have seen above, the minimum
is attained when $X\sim\mathrm{Uniform}(\tilde{B}_{2,n})$.

\bibliographystyle{plain}
\bibliography{bib}

\end{document}